\documentclass[a4paper,10pt,parskip=full]{amsart}
\pdfoutput=1

\usepackage[utf8x]{inputenc}

\usepackage{amsmath, amssymb, amsthm}
\usepackage{dsfont}

\usepackage{graphicx, grffile} % use grffile to allow for proper file extension deduction
\usepackage{pgf}
\usepackage[export]{adjustbox}
\usepackage{color}

\usepackage[hidelinks]{hyperref}
\usepackage{lmodern} % to stabilize compatibility with pgf plot fonts

\usepackage[protrusion=true,expansion=true]{microtype}

\newcommand{\cTold}{\mathcal{T}^\textnormal{old}}
\newcommand{\Edges}{\mathcal{E}}
\newcommand{\Element}{\tau}
\newcommand{\op}[1]{\operatorname{#1}}
\newcommand{\st}{\,|\,}
\newcommand{\Bigst}{\,\Big|\,}

\renewcommand{\d}{\,\textnormal{d}}
\renewcommand{\o}{\textnormal{old}}

\newcommand{\cT}{\mathcal{T}}
\newcommand{\cS}{\mathcal{S}}
\newcommand{\cN}{\mathcal{N}}

\newcommand{\cQ}{\mathcal{Q}}
\newcommand{\cD}{\mathcal{D}}
\newcommand{\cV}{\mathcal{V}}
\newcommand{\cG}{\mathcal{G}}
\newcommand{\cL}{\mathcal{L}}
\newcommand{\cJ}{\mathcal{J}}
\newcommand{\cW}{\mathcal{W}}
\newcommand{\cP}{\mathcal{P}}

\newcommand{\Wl}{\mathbb{W}}

\newcommand{\R}{\mathbb{R}}
\newcommand{\Rinfty}{\mathbb{R}\cup \{\infty\}}
\newcommand{\N}{\mathbb{N}}

\renewcommand{\a}{\alpha}
\newcommand{\eps}{\varepsilon}
\newcommand{\toladapt}{\op{Tol}_{\mathrm{adapt}}}
\newcommand{\tolcoarsen}{\op{Tol}_{\mathrm{derefine}}}

\newcommand{\tolcorrection}{\op{Tol}_{\mathrm{correction}}}

\newcommand{\norm}[2][]{\left\lVert#2\right\rVert_{#1}}
\newcommand{\normc}[1]{\norm[\mathrm{\phi}]{#1}}
\newcommand{\normw}[1]{\norm[{\mathrm{\theta}}]{#1}}

\newcommand{\euclideaninner}[2]{\langle #1, #2 \rangle}

\newcommand{\latentheat}{L}

\newcommand{\kin}{\beta}
\newcommand{\conv}{h_c}

\providecommand{\inlineargmin}[1]{\op{arg\,min\,}_{#1}}

\newtheorem{theorem}{Theorem}[section]

\newtheorem{proposition}{Proposition}[section]

\newtheorem{problem}{Problem}[section]

\usepackage{fancyhdr}
\title[Multi-phase Penrose--Fife systems]{Numerical approximation of multi-phase Penrose--Fife systems}
\author[Gr\"aser]{Carsten Gr\"aser}
\address{Carsten Gr\"aser\\
Freie Universit\"at Berlin\\
Institut f\"ur Mathematik\\
Arnimallee~6\\
D - 14195~Berlin\\
Germany}
\email{graeser@math.fu-berlin.de}

\author[Kahnt]{Max Kahnt}
\address{Max Kahnt\\
Freie Universit\"at Berlin\\
Institut f\"ur Mathematik\\
Arnimallee~6\\
D - 14195~Berlin\\
Germany}
\email{max.kahnt@math.fu-berlin.de}

\author[Kornhuber]{Ralf Kornhuber}
\address{Ralf Kornhuber\\
Freie Universit\"at Berlin\\
Institut f\"ur Mathematik\\
Arnimallee~6\\
D - 14195~Berlin\\
Germany}
\email{kornhuber@math.fu-berlin.de}

\date{\today}

\begin{document}

\begin{abstract}
 We consider a non-isothermal multi-phase field model. 
 We subsequently discretize implicitly in time and with linear finite elements.
 The arising algebraic problem is formulated in two variables
 where one is the multi-phase field, 
 and the other contains the inverse temperature field.
 We solve this saddle point problem numerically by a non-smooth Schur--Newton approach
 using truncated non-smooth Newton multigrid methods.
 An application in grain growth as occurring in liquid phase crystallization of silicon is considered.
\end{abstract}

\maketitle

\section{Introduction}

% Einordnung: 
% Phase transitions
The mathematical modelling of phase transitions has a long history 
and has stimulated new developments in the field of variational inequalities
and free boundary value problems over more than three decades~\cite{BaiocchiCapelo1984,Crank1982,kinderlehrer1980introduction}.
% Stefan problems and phase field
Particular attention was paid to problems of Stefan-type~\cite{Visintin2012} 
and their mathematical description by phase field models~\cite{BrokateSprekels96}.
In this approach, phase transitions are represented
by an order parameter that is strongly varying across the (diffuse) interface.
The evolution of the order parameter is typically obtained 
from some gradient flow of a suitable Ginzburg--Landau free energy
that provides non-decreasing entropy (thermodynamical consistency) and
could be mass conserving (phase separation) or non-conserving (phase transition).
% extension of Penrose--Fife by Stinner et al.
More recently Stinner et al.~\cite{Stinner2004} extended
well-established thermodynamically consistent, 
two-phase Penrose--Fife models~\cite{BrokateSprekels96} 
to multiple phases (non-conserved) and components (conserved).
Existence of solutions to the resulting balance equations 
for the energy, order parameters, and concentrations of components
was studied in \cite{stinner2007weak}.

% Numerical analysis and 
While the numerical analysis of two-phase Penrose--Fife models 
was based on implicit time discretization~\cite{OKlein97},
previous numerical computations with multiple phases and components 
were typically based on an explicit approach~\cite{nestler2008phase}.
In this way, the solution of non-smooth, large-scale algebraic systems
is avoided at the expense of severe stability constraints on the time step.

% in this paper
In this paper, we consider a multi-phase extension of 
the classical Penrose--Fife system~\cite{BrokateSprekels96,OKlein97,penrose1990thermodynamically}.
Following Stinner et al.~\cite{Stinner2004}, this system is 
derived from a general entropy functional that combines a Ginzburg--Landau energy 
with the thermodynamic entropy.
We concentrate on a numerical approach based on semi-implicit time discretization 
(with explicit treatment of the concave terms~\cite{deckelnick2005computation,graser2013time})
and first-order Taylor approximation of nonlinearities associated with inverse temperature.
Variational arguments are used to show  the existence and uniqueness
of solutions of the resulting spatial problems
and the thermodynamic consistency of this time discretization.
Spatial discretization is performed by piecewise linear finite elements
with adaptive mesh refinement based on hierarchical a 
posteriori error estimation~\cite{graeser:2014a,GraeserKornhuberSack2010}. 
The resulting large-scale non-smooth algebraic systems are solved by
non-smooth Schur--Newton multigrid (NSNMG) methods~\cite{graeserthesis,graeser2014,graser2009truncated} 
exploiting again the saddle point structure of these problems.
In our numerical experiments, we observe optimal order of convergence
of the spatial discretization  and mesh-independent, fast convergence speed
of NSNMG
with nested iteration. Furthermore, 
our computations suggest that non-decreasing entropy is preserved under the spatial discretization.
Application to a liquid phase crystallization (LPC) process 
occurring in the fabrication of thin film silicon solar cells~\cite{kuhnapfel2015towards}
underline the potential of the presented solution approach.

%%%%%%%%%%%%%%%%%%%%%%%%%%%%%%%%%%%%%%%%%%%%%%%%%%%%%%%

\section{Phase field modelling}\label{ss:assumptions}

\subsection{Thermodynamical background}
Let $\Omega \subset \R^d$, $d=1$, $2$, $3$, 
be a bounded domain with Lipschitz boundary $\Gamma= \partial \Omega$.
Following \cite{Stinner2004}, we consider the entropy functional
\begin{align}
 S(e,\phi) = \int_\Omega s(e,\phi) - (\tfrac\eps2 \gamma^2(\phi, \nabla\phi) + \tfrac1\eps \psi(\phi)) \d x,
\end{align}
where the entropy density $s$ depends on the internal energy density $e$ 
and on the multi-phase field $\phi= (\phi_{\a})_{\a=1}^M$, $\gamma$ represents the surface gradient entropy,
and $\psi$ a multi-well potential with $M$ distinct minima. 
The components of $\phi$ describe relative fractions of a given substance.
Hence, it is natural to impose the algebraic constraint
\begin{equation} \label{eq:ONE}
\sum_{\a=1}^M \phi_\a =1.
\end{equation}
We postulate the Gibbs relation~\cite{AltPawlow,Stinner2004}
\begin{align}\label{Gibbs-relation}
  \d f &= -s \d T + \sum_{\a=1}^M  f_{,\phi_\a}  \d \phi_\a 
\end{align}
with  absolute temperature $T>0$ and a Helmholtz free energy density $f=f(T,\phi)$ according to
\begin{equation} \label{energy-contributions}
  e = f + sT.
\end{equation}
As a consequence, we have
\begin{equation}\label{eq:entropytemp}
  s=-f_{,T},\qquad 
 \d s  =  \tfrac{1}{T}(\d e - \d  f)  =   \displaystyle \tfrac{1}{T} \d e- \sum_{\a=1}^M    \tfrac{1}{T} f_{, \phi_{\a}}  \d \phi_\a ,
\end{equation}
and therefore
\begin{equation}\label{eq:ENTROPDIFF}
 s_{, e}= \tfrac{1}{T},\qquad s_{,\phi_\a}= - \tfrac{1}{T} f_{, \phi_\a} \quad \alpha =1, \dots, M.
\end{equation}
We assume that the free energy density $f$
is obtained by interpolation of the individual bulk free energies
$L_\a \tfrac{T-T_\a}{T_\a} - c_vT(\ln(T)-1)$
for each phase $\a$. Here,  $\latentheat_\a \geq 0$ and $T_\a>0$ represent the latent heat and
melting temperature of pure phase $\a$, respectively, and
$c_v>0$ is the specific heat capacity.
In the light of \eqref{eq:ONE}, this leads to
\begin{equation}\label{total-free-energy}
    f(T,\phi) = \sum_{\a=1}^M  L_\a \tfrac{T-T_\a}{T_\a}\phi_\alpha - c_vT(\ln(T)-1),
\end{equation}
and we have
\begin{equation}\label{individual-free-energy}
    f_{, \phi_\a}(T,\phi) = L_\a \tfrac{T-T_\a}{T_\a}.
\end{equation}

Utilizing  \eqref{energy-contributions}, \eqref{eq:entropytemp}, and  the state equation \eqref{total-free-energy}, 
we can represent the entropy $s$ and the energy $e$
in terms of the temperature $T$ and the phase field $\phi$ according to 
\begin{equation} \label{energy}
    s(e,\phi) = \tilde{s}(T,\phi) = - \sum_{\a = 1}^M L_\a \tfrac{1}{T_\a}\phi_\a + c_v \ln(T),  \quad    
e = \tilde{e}(T,\phi) = - \sum_{\a=1}^M  L_\a\phi_\a +c_v T.
\end{equation}
Though our approach could be extended to anisotropic interfacial energies~\cite{graser2013time}, we choose
\begin{equation} \label{eq:INTERFACE}
 \gamma(\phi,\nabla\phi) = \vert \nabla\phi\vert
\end{equation}
for simplicity. Finally, $\psi$ stands for the classical multi-obstacle potential \cite{blowey_elliot:ch_obst_math:1991,barrett1997finite}
\begin{equation}\label{well-potential}
  \psi(\phi) = \chi_G(\phi) + \tfrac12\phi^T K \phi,
\end{equation}
with $\chi$ denoting the characteristic function 
\begin{equation}
    \chi_A(x) = 0 \textnormal{ if } x \in A \textnormal{ and } \chi_A(x) = \infty \textnormal{ if } x \not\in A
\end{equation}
and $G$ the Gibbs simplex
\begin{equation} \label{eq:GIBBSSIMP}
    G=\{v=(v_\a)\in \R^M \st
    \sum_{\a =1}^M v_\a = 1 \text{ and }v_\a \geq 0, \; \a=1,\dots,M\} \subset \R^M.
\end{equation}
We choose the negative definite interaction matrix
\[
    K=- I \in \R^{M,M}
\]
leading to the concave contribution $\phi^T K\phi = - |\phi|^2$ to the multi-obstacle potential.

With these specifications the entropy functional takes the form
\begin{align}\label{eq:entropy_with_obstacle}
 S(e,\phi) =S_0(e,\phi)  - \chi_\cG(\phi), \quad S_0(e,\phi)= \int_\Omega s(e,\phi) - \tfrac{\eps}{2}|\nabla\phi|^2 - \tfrac{1}{2\eps} \phi^TK\phi \d x
\end{align}
where $\chi_\cG$ is the characteristic functional of
\begin{equation}
    \cG=\{v \in H^1(\Omega)^M \st v(x)\in G\text{ a.e. in }\Omega\}.
\end{equation}

\subsection{A multi-phase Penrose--Fife system}

We postulate the continuity equation
\begin{equation} \label{eq:ECONT}
e_t=-\nabla \cdot  J_0(e,\phi) + q(e,\phi)
\end{equation}
with the flux 
\[
J_0(e,\phi) =   
\kappa  \nabla (\delta_e S)(e, \phi),
\]
mobility $\kappa>0$, variational derivative $\delta_e S$, and  a source term $q(e,\phi)$ to obtain
\begin{equation} \label{energy-balance}
e_t=  - \nabla \cdot \kappa  \nabla ( \delta_e  S ) (e,\phi)  + q(e, \phi).
\end{equation}

We assume that the outward energy flux is proportional to the difference of the temperature $T$ and 
a given boundary temperature $T_\Gamma$ or, more precisely, we prescribe 
\begin{equation} \label{eq:ENGBND}
J_0 \cdot n =   \conv (T -T_\Gamma)
\end{equation}
with the convection coefficient $\conv>0$ and the outward normal $n$ to $\Omega$.
It is also convenient to introduce the inverse temperature
\begin{equation} \label{eq:TTRANS}
\theta = \frac{1}{T}, \qquad \theta_\Gamma = \frac{1}{T_\Gamma}.
\end{equation}
Note that for given $\phi$ the variables  $e$, $T$, and $\theta$ can be transformed into each
other due to the strictly monotone relationships \eqref{energy} and \eqref{eq:TTRANS}.

In order to provide a non-decreasing entropy $S(e,\phi)$ in the course of the phase evolution, we set
\begin{equation} \label{maximize-entropy}
\eps\kin  \phi _t \in  \delta_\phi S_0(e,\phi) - \partial \chi_\cG(\phi) ,
\end{equation}
with a kinetic coefficient $\kin>0$,
the variational derivative $\delta_\phi S_0(e,\phi)$,
and the subdifferential $\partial \chi_\cG(\phi)$ of the convex functional $\chi_\cG$.
For the phase field we impose homogeneous Neumann boundary conditions
\begin{equation}\label{eq:PHASEBND}
\frac{\partial}{\partial n}\phi_\alpha = 0, \quad \alpha =1, \dots, M.
\end{equation}
Utilizing \eqref{energy}, \eqref{eq:INTERFACE}, \eqref{well-potential}, 
and the transformation \eqref{eq:TTRANS},
a weak formulation of the differential equations \eqref{energy-balance} and
\eqref{maximize-entropy} with  boundary conditions 
\eqref{eq:ENGBND} and \eqref{eq:PHASEBND}, respectively, reads as follows.

\begin{problem}[Multi-phase Penrose--Fife system with obstacle potential] \ \\%
 \label{prob:PF}%
Find the phase field
$\phi \in  L^2( (0, t^\ast), H^1(\Omega)^M) \cap H^1((0, t^\ast), L^2(\Omega)^M)$ and positive
inverse temperature
 $\theta \in L^2((0, t^\ast), H^1(\Omega)) \cap H^1((0, t^\ast), L^2(\Omega))$
such that
\begin{equation} \label{eq:INIT}
\phi(\cdot,0) = \phi^0,\qquad \theta(\cdot,0) = \theta^0
\end{equation}
holds with given initial conditions  $\phi^0\in L^2(\Omega)^M$, $\theta^0\in L^2(\Omega)$, 
$\theta^0 \geq 0$ a.e.\ in $\Omega$
and
\begin{subequations}  \label{pvi}
\begin{align} %
    \displaystyle  (\eps\kin \phi_t + \widetilde{\latentheat} - \theta \latentheat +\tfrac1\eps K\phi,v-\phi) 
 + \eps(\nabla \phi, \nabla (v-\phi)) + \chi_\cG(v)-\chi_\cG(\phi) &\geq 0 \label{pvi:1} \\
 \left(  -L^T\phi_t + c_v \tfrac{1}{\theta^2}\theta_t - q, w\right)
  - (\kappa\nabla \theta ,  \nabla w) + \left(\conv \left(\tfrac{1}{\theta} - \tfrac{1}{\theta_\Gamma}\right), w\right)_{\Gamma} &= 0
 \label{pvi:2}
\end{align}%
\end{subequations}%
holds for all $v\in H^1(\Omega)^M$ and $w \in H^1(\Omega)$.
\end{problem}

Here, $(0,t^\ast)\in \R$ is the considered time interval,
$\latentheat=(\latentheat_\alpha)_{\alpha=1}^M$ and 
$\widetilde{\latentheat}=(\frac{\latentheat_\alpha}{T_\a})_{\alpha=1}^M$ 
are constant vectors used to simplify the notation of~\eqref{eq:ENTROPDIFF} and \eqref{total-free-energy}, 
and $(\cdot, \cdot)$, $(\cdot, \cdot)_{\Gamma}$ stand for the scalar product in $L^2(\Omega)$,
$L^2(\Gamma)$, respectively.

For existence and uniqueness results 
in the special case $M=2$, we refer to  \cite[Section 7.2]{BrokateSprekels96},
\cite{OKlein97}, and the references cited therein.

\begin{proposition} \label{prop:THERMOCONSISTENCY}
The multi-phase Penrose--Fife system is thermodynamically consistent in the sense that 
\begin{equation} \label{eq:TCC}
S(e(t), \phi(t))\geq S(e(t_0), \phi(t_0)) \quad \forall  t \in [t_0,t^*]\subset (0,t^*]
\end{equation}
holds for any solution $(\phi, \theta)$ of Problem~\ref{prob:PF} with $q=0$, $\conv=0$
satisfying $\phi \in C^1([t_0, t^*], H^1(\Omega)^M)$ for $[t_0,t^*]\subset (0,t^*]$.
\end{proposition}
\begin{proof}
    Since $\phi \in \cG$ for almost all $t$ and in view of \eqref{energy} and \eqref{eq:entropy_with_obstacle} we can write
    \begin{align}\label{eq:S_hat}
        S(e,\phi) = \hat S(\theta,\phi) = \int_\Omega -\widetilde{\latentheat}^T \phi + c_v \ln(\tfrac1\theta) - \tfrac\eps2 \lvert \nabla\phi \rvert^2 - \tfrac1{2\eps} \phi^T K \phi \d x.
    \end{align} 
    Testing \eqref{pvi:1} with $v = \phi(t-\tau)$, dividing by $\tau>0$, and letting $\tau \to 0$
    we get
    \begin{align*}
        0 \leq \eps \Vert \sqrt{\kin}\phi_t \Vert^2
            \leq ( - \widetilde{\latentheat} + \theta \latentheat - \tfrac1\eps K\phi, \phi_t) - \eps(\nabla\phi, \nabla\phi_t)
    \end{align*}
    whereas testing \eqref{pvi:2} with $w=\theta$ yields
    \begin{align*}
        0 \leq (\kappa \nabla\theta, \nabla \theta) =
        (-\latentheat^T\phi_t + c_v \tfrac1{\theta^2}\theta_t, \theta).
    \end{align*}
    Adding both we get
    \begin{align*}
        0
        \leq ( - \widetilde{\latentheat}, \phi_t) - ( \tfrac{c_v}{\theta}, \theta_t) - \eps(\nabla\phi, \nabla\phi_t) - \tfrac1\eps(K\phi, \phi_t)
        = \langle \nabla \hat S(\theta, \phi)), (\theta_t, \phi_t)\rangle.
    \end{align*}
    Now integrating over $[t_0,t]$ provides the assertion.
\end{proof}

\subsection{Thin film approximation}
We consider a domain of the form $\Omega = \Omega' \times (0,H) \subset \R^d$, $d = 2,3$,
with a bounded Lipschitz domain $\Omega'\subset \R^{d-1}$ and $H>0$.
We assume that $\Omega$ is "thin"  in the sense that 
variations of $\phi$, $\phi^0$, $\theta$, $\theta^0$, and $q$ normal to $\Omega'$ as well as
the flux $J_0$ across $\partial \Omega'\times (0,H)$ can be neglected
and $J_0\cdot n(\cdot,0) = J_0\cdot n (\cdot,H)$ holds a.e.\ in $\Omega'$.
Inserting these assumptions into \eqref{eq:INIT}, \eqref{pvi}, 
we obtain the following thin film approximation of Problem~\ref{prob:PF}.

\begin{problem}[Thin film multi-phase Penrose--Fife system] \ \\%
 \label{prob:PFTHIN}%
Find the phase field 
$\phi \in  L^2( (0, t^\ast), H^1(\Omega')^M) \cap H^1((0, t^\ast), L^2(\Omega')^M)$ 
and positive inverse temperature
 $\theta \in L^2((0, t^\ast), H^1(\Omega')) \cap H^1((0, t^\ast), L^2(\Omega'))$
such that
\begin{equation}
\phi(\cdot,0) = \phi^0 ,\qquad \theta(\cdot,0) = \theta^0
\end{equation}
holds with given initial conditions  $\phi^0 \in L^2(\Omega)^M$, $\theta^0 \in L^2(\Omega)$, 
$\theta^0 \geq 0$ a.e.\ in $\Omega$
and
\begin{subequations} \label{pvi2D}
\begin{align}%
    \displaystyle  (\eps\kin \phi_t + \widetilde{\latentheat} - \theta \latentheat +\tfrac1\eps K\phi,v-\phi) 
 + \eps(\nabla \phi, \nabla (v-\phi)) + \chi_\cG(v)-\chi_\cG(\phi) &\geq 0 \label{pv2Di:1} \\
 \left(  -L^T\phi_t + c_v \tfrac{1}{\theta^2}\theta_t +\conv'\tfrac1\theta - \conv'\tfrac1{\theta_\Gamma} -q, w\right)
  - (\kappa\nabla \theta ,  \nabla w)  &= 0
 \label{pvi2D:2}
\end{align}%
\end{subequations}%
and $\conv'=\tfrac{2\conv}{H}$
for all $v\in H^1(\Omega')^M$ and $w \in H^1(\Omega')$.
\end{problem}

Here, $(\cdot, \cdot)=(\cdot, \cdot)_{\Omega'}$  stands for the  scalar product in $L^2(\Omega')$
for ease of notation.
Note that Problem~\ref{prob:PFTHIN} is essentially a 
$(d-1)$-dimensional analogue of Problem~\ref{prob:PF} with similar mathematical properties.
For example, the thermodynamic consistency in the sense of
Proposition~\ref{prop:THERMOCONSISTENCY}  is still valid
and all considerations concerning  the discretization and algebraic
solution of discretized problems to be reported below
carry over from Problem~\ref{prob:PF} to its thin film approximation.
For this reason, in the remainder we consider the more general equation
\begin{align}
(  -L^T\phi_t + c_v \tfrac{1}{\theta^2}\theta_t +h_{\Omega}\tfrac1\theta - q', w)
  - (\kappa\nabla \theta ,  \nabla w)  + (h_\Gamma(\tfrac1\theta - \tfrac1{\theta_\Gamma}), w)_\Gamma = 0
  \label{eq:therm-general}
\end{align}
with  the coefficients $h_\Omega, h_\Gamma \geq 0$, of which exactly one is zero
to recover Problems~\ref{prob:PF} resp.~\ref{prob:PFTHIN}.
Specifically, choose
$h_\Omega =0$, $h_\Gamma = \conv$, and $q'=q$ to obtain \eqref{pvi:2} and
$h_\Omega =\conv'$, $h_\Gamma = 0$, and $q'= q + \conv'\tfrac1{\theta_\Gamma}$ to obtain \eqref{pvi2D:2}.

\section{Discretization}
\label{section:discretization}
In this section we present a discretization of Problem~\ref{prob:PF} and Problem~\ref{prob:PFTHIN}
using the general equation \eqref{eq:therm-general}
by  Euler-type discretizations in time and finite elements in space.
Since an efficient approximation of the phase field
$\phi$ requires time-dependent,  locally refined spatial grids,
it is convenient to use Rothe's method~\cite{bornemann:adapt_multilevel_parabolic:1991},
i.e., the variational problem \eqref{pvi} is first discretized in time 
and the resulting spatial problems are then discretized in space, 
independently from each other.

\subsection{Implicit time discretization}

In light of  the well-known stiffness of the non-linear parabolic system of equations,
we use a semi-implicit Euler method.
More precisely, 
after approximating the time derivatives $\phi_t, \theta_t$ by backward finite differences with step size $\tau>0$, %$(*)$,
the nonlinearities $1/\theta$, $1/\theta^2$ are approximated by first-order Taylor expansion
(cf., e.g., \cite[Section 6.4]{BornemannDeuflhard2008})
\[
    \tfrac1{\theta(t)} 
    \doteq \tfrac{2}{\theta(t-\tau)} - \tfrac{\theta(t)}{(\theta(t-\tau))^2},\qquad
    \tfrac1{\theta(t)^2}
    \doteq \tfrac{3}{\theta(t-\tau)^2} - \tfrac{2\theta(t)}{\theta(t-\tau)^3}.
\]
In particular, this leads to
\[
\tfrac{1}{\theta(t)^2} \theta_t(t)
\doteq \tfrac{1}{\theta(t)^2 }\tfrac{\theta(t)-\theta(t-\tau)}{\tau}
=\tfrac{1}{\tau}\Bigl(\tfrac{1}{\theta(t)} - \tfrac{\theta(t-\tau)}{\theta(t)^2} \Bigr)
\doteq \tfrac{1}{\tau} 
\Bigl(\tfrac{\theta(t)}{\theta(t-\tau)^2} - \tfrac{1}{\theta(t-\tau)} \Bigr).
\]
Only the concave term $\frac{1}{\eps}K\phi$ is taken explicitly~\cite{deckelnick2005computation,graser2013time},
trading unconditional stability for a potential loss of accuracy, 
c.f.~\cite{blank2013primal} or ~\cite[section 6.3.1]{bartels2015numerical}.

For simplicity, we utilize the uniform time step size $\tau= t^*/n^*$ with given $n^*\in \N$,
and denote  the approximations of  $\phi(t_n)$, $\theta(t_n)$ at $t_n=n\tau$, $n=1,\dots,n^*$ by  $\phi^n$, $\theta^n$,
respectively. The spatial problem to be solved in 
the $n$-th time step then reads as follows.

\begin{problem}[Spatial multi-phase Penrose--Fife system with obstacle potential]\ \\%
\label{prob:PFSPAT}%
Find the phase field
$\phi^n \in  H^1(\Omega)^M$ and positive inverse temperature $\theta^n \in H^1(\Omega)$
such that
\begin{subequations}  \label{pviTD}
    \begin{align}
        a(\phi^n,v-\phi^n) + \chi_\cG(v) - \chi_\cG(\phi^n) 
        + b(v-\phi^n,\theta^n) &\geq \ell_1^n(v-\phi^n) \label{pviTD1}\\
        b(\phi^n,w) - c^n(\theta^n,w) &= \ell_2^n(w) \label{pviTD2}
    \end{align}%
\end{subequations}%
holds with the bilinear forms
\begin{subequations}\label{eq:BIFO}
    \begin{align}
        a(v,v') &= \eps(\kin v,v') + \eps\tau\left(\nabla v,\nabla v'\right), \\
        b(v,w) &= -\tau(L^T v,w),\\
        c^n(w,w') &= \tau(\tfrac{c_v + \tau h_\Omega}{(\theta^{n-1})^2} w,w' )
        +  \tau^{2} (\tfrac{h_\Gamma}{(\theta^{n-1})^2} w,w' )_\Gamma
        +  \tau^{2} (\kappa\nabla w, \nabla w' )
\intertext{and the linear functionals}
         \ell_1^n(v)&= (\eps \kin \phi^{n-1} - \tau  \widetilde{\latentheat}  - \tfrac{\tau}{\eps} K\phi^{n-1},v),  \label{eq:ell1}\\
        \ell_2^n(w)&= (\tau^{2} q' - \tau L^T \phi^{n-1} - \tau\tfrac{c_v + 2\tau h_\Omega}{\theta^{n-1}},w) - \tau^{2} h_\Gamma (  \tfrac{2}{\theta^{n-1}} -  \tfrac{1}{\theta_\Gamma} ,w )_\Gamma
    \end{align}
\end{subequations}
defined for all $v, v'\in H^1(\Omega)^M$ and $w,w' \in H^1(\Omega)$.
\end{problem}

\begin{proposition} \label{prop:THERMOCONSISTENCY-disc}
The time-discrete multi-phase Penrose--Fife system is thermodynamically consistent
in the sense that 
\begin{equation} \label{eq:TCC-disc}
S(e^n, \phi^n)\geq S(e^{n-1}, \phi^{n-1}) 
\end{equation}
holds for any solution $(\phi^n, \theta^n)$ of Problem~\ref{prob:PFSPAT} with $q=0$, $\conv=0$.
Here, $e^n = \tilde{e}(\tfrac{1}{\theta^n},\phi^n)$ is defined according to~\eqref{energy}.
\end{proposition}
\begin{proof}

Testing equation~(\ref{pviTD1}) with $v=\phi^{n-1}$ yields
\begin{align*}
\tfrac\eps2\vert\nabla\phi^n\vert^2 - \tfrac\eps2\vert\nabla\phi^{n-1}\vert^2 &\leq ( \widetilde\latentheat - \latentheat\theta^n + \tfrac1\eps K\phi^{n-1}, \phi^{n-1}-\phi^n)  
 \end{align*}
and testing equation~(\ref{pviTD2}) with $w=\theta^n$ for $q=0, \conv=0$ yields
\begin{align*}
    0 \leq  \tau \kappa\vert \nabla \theta^n \vert^2 &= (\latentheat\theta^n, \phi^{n-1}-\phi^n) + c_v\int_\Omega \tfrac{\theta^n}{\theta^{n-1}}(1-\tfrac{\theta^n}{\theta^{n-1}})\d x.\end{align*}
Adding both inequalities, we obtain
\begin{multline*}
    - ( \widetilde\latentheat, \phi^{n-1}-\phi^n)  - \tfrac\eps2(\vert\nabla\phi^{n-1}\vert^2 - \vert\nabla\phi^n\vert^2)  \\
    \leq  \,\, c_v\int_\Omega \tfrac{\theta^n}{\theta^{n-1}}(1-\tfrac{\theta^n}{\theta^{n-1}})\d x + \tfrac1\eps(K\phi^{n-1},\phi^{n-1} - \phi^n)
\end{multline*}
which, using the representation~\eqref{eq:S_hat}, finally can be used to show
\begin{align*}
S(e^{n-1},\phi^{n-1}) - S(e^n,\phi^n) 
  \leq
  \int_\Omega c_v (\ln(\tfrac{\theta^{n}}{\theta^{n-1}})+\tfrac{\theta^n}{\theta^{n-1}}(1-\tfrac{\theta^n}{\theta^{n-1}}))\d x.
\end{align*}
Since the right hand side is non-positive, this provides the assertion.
\end{proof}

It turns out that the system \eqref{pviTD}  can be regarded as optimality conditions
for a Lagrange-functional.

\begin{proposition} \label{prop:timediscrete_saddle}
    Problem~\ref{prob:PFSPAT} is equivalent to find 
    $\phi^n \in  H^1(\Omega)^M$ and  $\theta^n \in H^1(\Omega)$ such that
    \begin{equation}
        \cL^n(\phi^n,w)
            \leq \cL^n(\phi^n, \theta^n)
            \leq \cL^n(v, \theta^n)
            \qquad \forall \;  v \in H^1(\Omega)^M,\; w \in H^1(\Omega),
    \end{equation}
    with the Lagrangian $\cL^n$ given by
    \begin{align}
        \label{eq:lagrange_functional}
        \cL^n(v,w) = \cJ^n(v) - \ell_1^n(v) + b(v,w) -\ell_2^n(w)
            - \tfrac{1}{2}c^n(w,w)
    \end{align}
    denoting $\cJ^n(v)=\tfrac{1}{2}a(v,v) + \chi_\cG(v)$.
\end{proposition}

\begin{theorem} \label{thm:timediscrete_existence}
    Let $\phi^{n-1} \in \cG$, $\theta^{n-1}\in H^1(\Omega)$ and 
    $\theta^{n-1} \geq  c$ a.e.\ in $\Omega$ with a positive constant $c$.
    Then the spatial Problem~\ref{prob:PFSPAT}  admits a unique solution.
\end{theorem}
\begin{proof}
    Under the given assumptions, $\ell_1^n$ and $\ell_2^n$ are bounded
    linear functionals and $a(\cdot,\cdot)$ is symmetric and coercive.
    Since $\cG$ is closed and convex the functional $\cL^n(\cdot,w)$
    is strictly convex, coercive, and lower semi-continuous for all 
    fixed $w \in H^1(\Omega)$.
    Thus, we can define the dual functional
    \begin{align*}
        h(w) = - \inf_{v \in H^1(\Omega)^M} \cL^n(v,w) = (\cJ^n)^*(\ell_1^n - b(\cdot,w)) + \tfrac{1}{2}c^n(w,w) + \ell_2^n(w).
    \end{align*}
    Here, $(\cJ^n)^*$ is the convex and continuous polar of $\cJ^n$.
    Notice that $h$ is also convex and continuous because $c^n(\cdot,\cdot)$ is symmetric and positive-definite.

    By integrability of $(\theta^{n-1})^2$ there must be
    a subset $\Omega' \subset \Omega$ with positive measure and a constant
    $c_2>0$ such that $\theta^{n-1}\leq c_2$ on $\Omega'$. As a consequence
    we get
    \begin{align*}
        c^n(w,w) \geq \tfrac{c_v \tau}{c_2^2}\|w\|_{L^2(\Omega')}^2 + \tau^2 \kappa \|\nabla w\|^2
    \end{align*}
    and therefore coercivity of $c^n(\cdot,\cdot)$.
    Hence, $h$ is coercive, continuous, and strictly convex 
    and thus has a unique minimizer.
    Existence and uniqueness now follows from
    the fact that $(\phi^n,\theta^n)$ is a solution of Problem~\ref{prob:PFSPAT},
    if and only if $\theta^n$ is a minimizer of $h$ and $\phi^n = \inlineargmin{v \in H^1(\Omega)^M} \cL^n(v, \theta^n)$.
\end{proof}

In general it is not clear if the uniform positivity of  the inverse temperature 
is preserved  by solving Problem~\ref{prob:PFSPAT}. We refer, however, to \cite{OKlein97} for such kind of results
in the scalar case $M=2$.

\subsection{Adaptive finite element discretization}\label{subsec:fe-discretization}
We will now consider the adaptive finite element discretization of the spatial
Problem~\ref{prob:PFSPAT} for an individual fixed time stepIn order to improve readability we will from now on
drop all superscripts $(\cdot)^n$  that identify the current time step. 
We will designate quantities from the previous time step by the
superscript $(\cdot)^\o$ whenever necessary.

\subsubsection{Finite element discretization}
In the following, we assume that $\cT$ is a simplicial grid that
is either conforming or obtained via local hanging node refinement
of a conforming initial grid.
We will discretize the spatial Problem~\ref{prob:PFSPAT}
with respect to the conforming first order finite element space
\begin{align}\label{eq:fe_space}
    \cS = \cS(\cT) = \Bigl\{ v \in C(\overline{\Omega}) \Bigst v|_\Element \text{ is affine } \forall \Element \in \cT\Bigr\}
        \subset H^1(\Omega).
\end{align}
Notice that $\cS$ has a uniquely defined nodal basis
$\{\lambda_p \st p \in \cN\}$ satisfying $\lambda_p(q) = \delta_{pq}$ for all $p,q \in \cN$
where $\cN$ is the set of non-hanging nodes of $\cT$.
If the grid is obtained by uniform or local hanging node refinement,
the non-conforming mesh hierarchy induces a natural
hierarchy of subspaces of $\cS$ that can be used in
geometric multigrid methods.
For a detailed discussion of finite element spaces
on hierarchies of non-conforming, locally refined grids
we refer to~\cite{graeserthesis,graeser:2014a}.

Note that linearity implies that the Gibbs constraint can be evaluated node-wise, i.e.,
\begin{equation}
    \cG \cap \cS^M = \{v \in \cS^M \st v(p)\in G\;\forall p \in \cN \}.
\end{equation}

\begin{problem}[Discrete spatial multi-phase Penrose--Fife system]\ \\%
    \label{prob:PFSPATDISC}%
    Find the phase field
    $\phi_\cT \in  \cS^M$ and positive inverse temperature $\theta_\cT \in \cS$
    such that
    \begin{subequations}
        \begin{align}
            a(\phi_\cT,v-\phi_\cT) + \chi_\cG(v) - \chi_\cG(\phi_\cT) + b(v-\phi_\cT,\theta_\cT) &\geq \ell_1(v-\phi_\cT),
                \label{pvi:PFSPATDISC1} \\
            b(\phi_\cT,w) - c(\theta_\cT,w) &= \ell_2(w)
                \label{pvi:PFSPATDISC2}
        \end{align}
    \end{subequations}%
    holds
    for all $v\in \cS^M$ and $w \in \cS$.
\end{problem}

Here, $\ell_1,\ell_2$, and $c(\cdot,\cdot)$ are defined as in Problem~\ref{prob:PFSPAT}
but with $\phi^{n-1}$ and $\theta^{n-1}$ replaced by their finite element approximations $\phi^\o$ and $\theta^\o$.
To avoid negative values for $\theta_\cT$ due to overshooting when $q'$ is not resolved by the grid,
we replace the latter by its interpolation in $\cS$.
Notice that $\phi^\o$ and $\theta^\o$ are finite element functions on a grid
$\cTold$. In case of adaptive refinement, $\cTold$ is usually different from $\cT$.
See~\cite{graeser:2014a} for a detailed discussion.

In analogy to its continuous counterpart,  Problem~\ref{prob:PFSPATDISC} 
can be written as a non-linear, non-smooth saddle point problem.

\begin{proposition} \label{prop:SADDLE}
    Problem~\ref{prob:PFSPATDISC} is equivalent to finding 
    $\phi_\cT \in  \cS^M$ and  $\theta_\cT \in \cS$ such that
    \begin{equation}
        \cL(\phi_\cT,w)
            \leq \cL(\phi_\cT, \theta_\cT)
            \leq \cL(v, \theta_\cT)
            \qquad \forall \;  v \in \cS^M,\; w \in \cS,
    \end{equation}
    with the Lagrangian $\cL$ given according to \eqref{eq:lagrange_functional}.
\end{proposition}

Existence and uniqueness are also inherited from the continuous case.

\begin{theorem}\label{thm:discrete:existence}
    Let $\phi^\o \in \cG$, $\theta^\o\in H^1(\Omega)$ and 
    $\theta^\o \geq  c$ a.e.\ in $\Omega$ with a positive constant $c$.
    Then the spatial Problem~\ref{prob:PFSPATDISC} admits a unique solution.
\end{theorem}

\begin{proof}
    The proof can be carried out using the same arguments as in the proof
    of Theorem~\ref{thm:timediscrete_existence}.
\end{proof}

\subsubsection{Hierarchical a posteriori error estimation}\label{subsubsection:APOSTERIORI}
As the phase field $\phi$ is expected to strongly vary across the phase boundaries,
spatial adaptivity based on a posteriori error estimates is mandatory.
Similarly, the consumption of heat by phase changes may lead to strongly varying $\theta$.
Hierarchical error estimates rely on the solution of local defect problems.
While originally introduced for linear elliptic problems
\cite{FABornemann_BErdmann_RKornhuber_1993b,PDeuflhard_PLeinen_HYserentant_1989a,HolstOvallSzypowski:2011,OCZienkiewicz_JPSRGago_DWKelly_1983a}
this technique was successfully extended to non-linear problems
\cite{REBank_RKSmith_1993a}, constrained minimization
\cite{RHWHoppe_RKornhuber_1994a, RKornhuber_1995b,RKornhuberQZou_2008,SiebertVeeser2007,
graeser_kornhuber_veeser_zou:hest_energy_obstacle:2011}
and non-smooth saddle point problems
\cite{GraeserKornhuberSack2010,graeser:2014a,graeserthesis}.

Following \cite{graeserthesis,GraeserKornhuberSack2010,graeser:2014a}, we now derive
an a posteriori error estimate by a suitable approximation of the
defect problem associated with the defect Lagrangian
\begin{align*}
    \cD(e_\phi,e_\theta) = \cL(\phi_\cT + e_\phi, \theta_\cT + e_\theta).
\end{align*}
In the first step the defect problem is discretized with respect to
a larger finite element space $\cQ^M\times \cQ$, where
$\cQ = \cS(\cT')$ is defined analogously to~\eqref{eq:fe_space} for the grid
$\cT'$ obtained by uniform refinement of $\cT$.
Note that we have $\cQ=\cS \oplus \cV$ with $\cV$ denoting the incremental space
\begin{align*}
    \cV= \op{span}\{\lambda_p' \st p \in \Edges \}
    \subset \op{span}\{\lambda_p' \st p \in \cN' \} = \cS'.
\end{align*}
Here,
$\cN'$ denotes the set of non-hanging nodes in $\cT'$,
$\{\lambda_p' \st p \in \cN'\}$ the nodal basis of $\cS'$,
and $\Edges=\cN' \setminus \cN$
is the set of all edge mid points in $\cT$ that are non-hanging in $\cT'$.

In the second step, the discrete defect problem is localized
by ignoring the coupling between $\cS$ and $\cV$ and also
the coupling between $\lambda_p'$ for all $p \in \Edges$.
Denoting $\cD_p(r,s) = \cD(r\lambda_p',s\lambda_p')$,
this results in the local saddle point problems
\begin{align*}
    (e_{\phi,p},e_{\theta,p}) \in \R^{M}\times \R: \;\;
    \cD_p(e_{\phi,p},s)
    \leq
    \cD_p(e_{\phi,p},e_{\theta,p})
    \leq
    \cD_p(r,e_{\theta,p})
    \;\;\forall (r,s) \in \R^M\times \R
\end{align*}
for all $p \in \Edges$ that give rise to the hierarchical a posteriori error estimate
\begin{equation}
    \eta=\Bigl(\sum_{p\in\Edges} \eta_p^2\Bigr)^{\frac12}, \qquad
    \eta_p^2=
        \normc{e_{\phi,p}\lambda_p'}^2 +
        \normw{e_{\theta,p}\lambda_p'}^2, \qquad
    p \in \Edges\label{eq:APOST}
\end{equation}
with the problem-dependent norms
\begin{align}
    \normc{v}^2 &=
    a(v,v), &
    \normw{w}^2 &= 
        c(w,w)
\end{align}
on $\cV^M$, $\cV$, respectively.

 \subsubsection{Adaptive mesh refinement}\label{subsubsec:adaptivity}

 The initial grid for the adaptive refinement should be sufficiently
 fine to detect basic features of the unknown spatial approximation
 and sufficiently coarse for efficiency of the overall adaptive procedure.
 The construction of such a grid starts with the grid $\cTold$ from the preceding time step.
 In the first time step, we select a suitable, uniformly refined grid $\cTold$.

 We begin by coarsening $\cTold$.
 To this end, we keep all simplices from the grid $\cTold$ from the preceding time step
 that were obtained by at most $j_{\op{min}}$ refinements.  In addition, we keep
 all simplices $\Element$ such that $\phi^{\o}$ exhibits a strong
 local variation on $\Element$ that is not visible after coarsening, i.e.,
 such that
 \begin{align*}
     \| |\nabla (I_{\Element} \phi^{\o}) | \|_{L^\infty(\Element)}\geq\tolcoarsen
     \qquad \textnormal{ and } \qquad
     \| |\nabla (I_{\Element'} \phi^{\o} ) | \|_{L^\infty(\Element')}<\tolcoarsen
 \end{align*}
 holds with $\Element'$ denoting the simplex resulting from coarsening of $\Element$.
$I_{\Element}$ and $I_{\Element'}$ are the linear interpolation operators
 to $\Element$ and $\Element'$, respectively.
 This set of simplices is completed by additional local refinements.
 Possible additional refinement is used to uniformly bound
 the ratio of diameters of adjacent simplices.

The adaptive mesh refinement of the resulting initial grid $\cT$
 is based on the local error indicators $\eta_p$ defined
 in \eqref{eq:APOST}. In each step, the indicators $\eta_{p_i}$, $i=1,\dots,|\Edges|$,
 are arranged with decreasing value, to determine the minimal number $i_0$ of indicators
 such that
 \begin{equation} \label{eq:REFTHRES}
    \sum_{i=1}^{i_0} \eta_{p_i}^2 \geqslant \rho\eta^2
 \end{equation}
 holds with a given parameter $\rho\in [0,1]$.
 Then all simplices $\Element \in \cT$ with the property $p_i\in \Element$
 for some $p_i$ with $i\leq i_0$
 are marked for refinement~\cite{Doerfler1996}.
 Each marked simplex is partitioned by (red) refinement~\cite{Bey2000,bornemann:1993}.
 Again,
 possible additional refinement is used to uniformly bound
 the ratio of diameters of adjacent simplices.
 The refinement process is stopped, when the
  estimated relative error is less than a given tolerance $\toladapt>0$, i.e., if
  \begin{equation}\label{eq:STOPCRIT}
      \eta < \toladapt\cdot
          \biggl( \normc{\phi_\cT}^2 + \normw{\theta_\cT}^2\biggr)^{\frac 1 2}.
  \end{equation}

\section{Algebraic solution}\label{section:solver}
Several methods have been proposed for the algebraic solution of
discretized multi-phase field equations. The solution of multi-phase
Allen--Cahn-type equations via primal--dual active set methods
was discussed in~\cite{BlankGarkeSarbuStyles2013} and multigrid methods
for such problems where proposed
in~\cite{GraeserSander2014_preprint,kornhuber_krause:mg_vector_ac:2006}.
In contrast to the second order problems with minimization structure considered
there, Problem~\ref{prob:PFSPATDISC} is a saddle point problem for a discretized
fourth order equation. For similar problems resulting from multi-component Cahn--Hilliard
systems block Gau\ss--Seidel-type algorithms
with component-wise and vertex-wise blocking where proposed
in~\cite{BloweyCopettiElliott1996} and~\cite{Nuernberg2009}, respectively.
To overcome the mesh-dependence of the Gau\ss--Seidel approach,
a nonsmooth Schur--Newton method was proposed in~\cite{graeser2014}.

While, for Cahn--Hilliard-type systems, the local sum constraint in $G$
can be enforced using the chemical potential as a natural Lagrange multiplier (cf.~\cite{graeser2014}),
Problem~\ref{prob:PFSPATDISC} is structurally different and the introduction
of such a multiplier would change the structure of the problem.
Therefore, we now introduce a nonsmooth Schur--Newton method
that does not involve  such a  multiplier.

\subsection{Matrix notation}
For the presentation of the algebraic solver for the iterative solution
of Problem~\ref{prob:PFSPATDISC} we first formulate this problem
in terms of coefficient vectors and matrices.
To this end let $N=\op{dim}\cS$ and introduce an enumeration of the nodes $\cN=\{p_1,\dots, p_N\}$.
To simplify the presentation,
we use the abbreviated notation $\lambda_k = \lambda_{p_k}$
for the nodal basis $\lambda_1,\dots,\lambda_N$
of $\cS$ and
introduce the basis $\lambda^1,\dots,\lambda^{MN}$ of $\cS^M$
where $\lambda^{\pi(i,k)} = b^i \lambda_{k}$,
$b^i \in \R^M$ is the $i$-th Euclidean basis vector, and $\pi : \lbrace 1,\ldots, N\rbrace \times \{ 1,\dots, M\}$
is the bijective index map given by
\begin{align*}
    \pi(k,i) = i+M(k-1).
\end{align*}

For $v\in \cS^M, w\in \cS$ we then get the associated coefficient vectors $V \in \R^{MN}, W \in \R^N$
\begin{align*}
    v &= \textstyle\sum_{i=1}^{MN} V_i \lambda^i,
    &
    w &= \textstyle\sum_{i=1}^N W_i \lambda_i.
\end{align*}
Using the matrices $A\in\R^{MN,MN}$, $B\in\R^{N,MN}$, $C\in \R^{N,N}$ and vectors $F\in \R^{MN}, G\in \R^{N}$
given by
\begin{align*}
    A_{ij} &= a(\lambda^j, \lambda^i), &
    B_{ij} &= b(\lambda^j, \lambda_i), &
    C_{ij} &= c(\lambda_j, \lambda_i), &
    F_{i} &= \ell_1(\lambda^i), &
    G_{i} &= \ell_2(\lambda_i),
\end{align*}
and the characteristic functional $\chi_{G_N}: \R^{MN} \to \Rinfty$ of
\begin{align*}
    G_N = \{V \in \R^{MN} \st (V_{\pi(k,i)})_{i=1,\dots,M} \in G \quad \forall k \}
        = \{V \in \R^{MN} \st \textstyle\sum_{i=1}^{MN} V_i \lambda^i \in \cG\}
\end{align*} 
Problem~\ref{prob:PFSPATDISC} can be written as:

\begin{problem}[Algebraic variational inequality]\ \\%
    \label{prob:algebraic_vi}%
    Find the coefficient vectors $\Phi \in \R^{MN}$ and $\Theta \in \R^N$ of $\phi_\cT$ and $\theta_\cT$, respectively,
    such that
    \begin{subequations}
        \begin{align}
            \euclideaninner{A \Phi}{V-\Phi} + \chi_{G_N}(V) - \chi_{G_N}(\Phi) + \euclideaninner{B^T \Theta}{V-\Phi} &\geq \euclideaninner{F}{V-\Phi},
                \label{eq:algebraic_vi_1} \\
                B \Phi - C\Theta &= G
                \label{eq:algebraic_vi_2}
        \end{align}
    \end{subequations}%
    holds for all $V \in \R^{MN}$.
\end{problem}

Problem~\ref{prob:algebraic_vi} can equivalently be written in operator
notation as a non-linear saddle point problem using in turn the subdifferential
of $\chi_{G_N}$.

\begin{problem}[Discrete saddle point problem]\ \\%
    \label{prob:discrete_spp}%
    Find the coefficient vectors $\Phi \in \R^{MN}$ and $\Theta \in \R^N$ of $\phi_\cT$ and $\theta_\cT$, respectively,
    such that
    \begin{align}
        \begin{pmatrix}
            A+\partial \chi_{G_N} & B^T\\
        B & -C
        \end{pmatrix}
        \begin{pmatrix}
            \Phi\\
            \Theta
        \end{pmatrix}
        \ni
        \begin{pmatrix}
            F\\
            G
        \end{pmatrix}.
    \end{align}
\end{problem}

For later reference we note that the Lagrangian for this saddle point problem
is given by the following discrete analogue
\begin{align*}
    L(V,W) = J(V) - \euclideaninner{F}{V} + \euclideaninner{BV - G}{W} - \tfrac{1}{2} \euclideaninner{CW}{W}
\end{align*}
of $\cL^n$ where
$J(V)=\tfrac{1}{2}\euclideaninner{AV}{V} + \chi_{G_N}(V)$
is the analogue of $\cJ^n$.

\subsection{Non-smooth Schur--Newton multigrid methods}

In the context of non-smooth Schur--Newton methods as introduced in~\cite{graeserthesis, graeser2014},
it is shown that problems of the form of Problem~\ref{prob:discrete_spp} can equivalently be formulated
as the following dual minimization problem.

\begin{problem}[Dual minimization problem]\ \\
    \label{prob:discrete_dual}%
    Find $\Theta \in \R^N$ such that
    \begin{align*}
        h(\Theta) \leq h(W) \qquad \forall W\in \R^N
    \end{align*}
    where $h:\R^N \to \R$ is the dual functional
    \begin{align*}
        h(W) &= - \inf_{V \in \R^{MN}} L(V,W) = -L(\Phi(W),W)
    \end{align*} 
    and $\Phi(W) = (A+\partial \chi_{G_N})^{-1}(F-B^T W)$.
\end{problem}

\begin{proposition}\label{prop:discrete_dual_problem}
    Problems~\ref{prob:discrete_spp} and \ref{prob:discrete_dual} are equivalent
    in the sense that $(\Phi,\Theta)$ solves Problem~\ref{prob:discrete_spp} if
    and only if $\Theta$ solves Problem~\ref{prob:discrete_dual} and $\Phi = \Phi(\Theta)$.
    Furthermore, the dual functional $h:\R^N \to \R$ is convex and continuously
    differentiable with the Lipschitz-continuous derivative
    \begin{align*}
        \nabla h(W)
            &= -B\Phi(W) + CW + G \\
            &= -B(A+\partial \chi_{G_N})^{-1}(F- B^TW) + CW + G.
    \end{align*} 
\end{proposition} 

The proof of Proposition~\ref{prop:discrete_dual_problem} can be done
analogously to the proof of \cite[Theorem~2.1]{GraeserKornhuber2009a}.
This proof also shows that $h$ can be written as
\begin{align*}
    h(W) = J^*(F-B^T W) + \tfrac{1}{2}\euclideaninner{CW}{W} + \euclideaninner{G}{W}
\end{align*}
where $J^* : \R^{MN} \to \R$ is the polar (or conjugate) functional of $J$
which is convex itself.
This especially shows that $h$ is a strongly convex functional
because $C$ is positive definite due to coercivity of the
associated bilinear form $c(\cdot,\cdot)$ (cf. proof of Theorem~\ref{thm:timediscrete_existence}).

As a consequence of Proposition~\ref{prop:discrete_dual_problem} we
can apply gradient-related descent methods of the form
\begin{align}\label{eq:descent_method}
    \Theta^{\nu+1} = \Theta^\nu + \rho_\nu D^\nu
\end{align}
where $D^\nu \in \R^N$ is a decent direction and $\rho_\nu$ a step size.
The non-smooth Schur--Newton method as introduced
in \cite{GraeserKornhuber2009a,graeserthesis, graeser2014}
is such a descent method where $D^\nu$ is taken to be
\begin{align}\label{eq:schur_newton_direction}
    D^\nu = - S_\nu^{-1} \nabla h(\Theta^\nu)
\end{align}
and $S_\nu \in \R^{N,N}$ is a generalized linearization
of the non-smooth, non-linear but Lipschitz continuous
Schur complement operator $-\nabla h$ at $\Theta^\nu$.

The derivation of $S_\nu$ essentially amounts
to deriving a generalized linearization of the operator $(A+\partial \chi_{G_N})^{-1}$
at $Y=F- B^T \Theta^\nu$. For simple component-wise
bound constraints is has been shown in \cite{GraeserKornhuber2009a}
that such a linearization is given by $(A_{\cW(X)})^+$.
Here, $A_{\cW(X)}$ is the restriction of $A$ to $\cW(X) \times \cW(X)$,
$\cW(X)$ is the maximal subspace such that $J$ is locally smooth in
\begin{align*}
    (X + \cW(X)) \cap U_X
\end{align*}
for some neighborhood $U_X$ of $X = (A+\partial\chi_{G_N})^{-1}(Y)$,
and $(\cdot)^+$ is the Moore--Penrose pseudoinverse or, equivalently,
the inverse of $A_{\cW(X)}: \cW(X) \to \cW(X)$.

For the simplex constraints in the present problem we will use
exactly the same approach. In the following we will outline
the construction of $A_{\cW(X)}$ for local simplex constraints
following \cite{GraeserSander2014_preprint}. To this end
we identify vectors $V \in \R^{MN}$
with block-structured vectors $\hat{V} \in (\R^M)^N$
such that $V_{\pi(k,i)} = (\hat{V}_k)_i$.

Due to the product structure
\begin{align*}
    G_N = \{V \in \R^{MN} \st \hat{V} \in G^N\},
\end{align*}
of the feasible set $G_N$ we can determine the subspace
$\cW(X)$ in each block individually. Hence $\cW(X)$
takes the form
\begin{align*}
    \cW(X) = \bigl\{ V \in \R^{MN} \st \hat{V} \in \prod_{k=1}^N \Wl(\hat{X}_k)\bigr\}
\end{align*}
where $\Wl(\hat{X}_k)$ is the maximal subspace where $\chi_G$
is locally smooth near $\hat{X}_k$.
As outlined in \cite{GraeserSander2014_preprint}
the local subspace $\Wl(\xi)$ is given by
\begin{align*}
    \Wl(\xi) = \op{span} \{ b^i - b^j \in \R^M \st 1 \leq i < j \leq M, \,\xi_i >0, \,\xi_j>0 \}
\end{align*}
for $\xi \in \R^M$. Since $\cW(X)$ is a product space the orthogonal projection
$\cP_{\cW(X)} : \R^{MN} \to \cW(X)$ is given by a block diagonal matrix
where the $k$-th diagonal block is the orthogonal projection $\cP_{\Wl(\hat{X}_k)} : \R^M \to \Wl(\hat{X}_k)$.
For an explicit representation of $\cP_{\Wl(\hat{X}_k)} \in \R^{M,M}$ we refer to
\cite{GraeserSander2014_preprint}. Using $\cP_{\cW(X)}$ we now get
\begin{align*}
    A_{\cW(X)} = \cP_{\cW(X)} A \cP_{\cW(X)}.
\end{align*}
Although the chain rule does in general not hold true for generalized Jacobians
in the sense of Clarke (see, e.g.~\cite{graeserthesis}), we define a generalized linearization of
the non-linear Schur complement operator $-\nabla h$ at $\Theta^\nu$
in an analogous manner by
\begin{align}\label{eq:linearized_schur_complement}
    S_\nu = B \Bigl(A_{\cW(\Phi(\Theta^\nu))}\Bigr)^+ B^T + C.
\end{align}

As a consequence of the convexity of $h$ we can show global convergence.
\begin{theorem}\label{thm:schur_newton_convergence}
    Assume that the step sizes $\rho_\nu$ are efficient
    (cf. \cite{ortega_rheinboldt:iterative_solution:1970, GraeserKornhuber2009a}),
    then the iterates produced by the descent method \eqref{eq:descent_method}
    with Schur--Newton directions \eqref{eq:schur_newton_direction} for
    $S_\nu$ given by \eqref{eq:linearized_schur_complement} converge
    to the solution $\Theta$ of Problem~\ref{prob:discrete_dual}
    for any initial guess $\Theta^0 \in \R^N$.
\end{theorem}
\begin{proof}
    Notice that the $S_\nu$ are uniformly bounded from above
    and below. Hence global convergence follows from
    \cite[Theorem~4.2]{GraeserKornhuber2009a}.
\end{proof} 

Efficient step sizes $\rho_{\nu}$ as required in Theorem~\ref{thm:schur_newton_convergence}
can be obtained by classical step size rules like, e.g., the Armijo rule or bisection.
Notice that it is not necessary to evaluate $S_\nu^{-1}$ exactly in
\eqref{eq:schur_newton_direction} because global convergence is
preserved as long as the approximation of $S_\nu^{-1}$ is sufficiently accurate.
Since the dual functional $h$ is strongly convex, one can also show
global linear convergence with a rate depending on the bounds for $S_\nu$
and the step size rule. For further details we refer to \cite{GraeserKornhuber2009a}.

During each iteration of the algorithm two types of subproblems have
to be solved. The evaluation of $-\nabla h(\Theta^\nu)$ requires
to compute $\Phi(\Theta^\nu) = (A+\partial \chi_{G_N})^{-1}(F-B^T \Theta^\nu)$.
This is equivalent to minimizing $J(\cdot) + \euclideaninner{B^T \Theta^\nu}{\cdot}$,
i.e., a convex minimization problem for a quadratic functional with local simplex constraints.
If the step size rule requires several trial steps
further problems of this type have to be solved for each evaluation
of $h$ or $\nabla h$.
These convex minimization problems can efficiently be solved using non-linear
multigrid methods \cite{kornhuber_krause:mg_vector_ac:2006,GraeserSander2014_preprint}.
More precisely the TNNMG method for simplex-constrained problems as proposed in
\cite{GraeserSander2014_preprint} allows to solve these problems with an
effective complexity of $O(M^2 N)$.
This method was used in all numerical
examples presented below.

The second type of subproblems are the linear problems \eqref{eq:schur_newton_direction}
for the symmetric positive definite operators $S_\nu$. Since each $S_\nu$ is a linear
Schur complement this is equivalent to solving the linear saddle point problem
\begin{align*}
    \begin{pmatrix}
        A_{\cW(\Phi(\Theta^\nu))} & (B\cP_{\cW(\Phi(\Theta^\nu))})^T \\
        B \cP_{\cW(\Phi(\Theta^\nu))} & -C
    \end{pmatrix} 
    \begin{pmatrix}
        \tilde{V}^\nu \\
        D^\nu
    \end{pmatrix}
    =
    \begin{pmatrix}
        0 \\
        \nabla h ( \Theta^\nu)
    \end{pmatrix}
\end{align*} 
whose solution is unique in $(\op{ker} \cP_{\cW(\Phi(\Theta^\nu))})^\perp \times \R^N$.
In the numerical examples presented below we used a linear multigrid method with
a Vanka-type smoother to solve these problems. Notice that there is no convergence
proof for this linear iterative method. To increase its robustness it can be
used as preconditioner for a GMRes iteration.

\section{Numerical experiments}
All our computations are based on a non-dimensionalized 
version of the  Penrose--Fife systems stated in 
Problem~\ref{prob:PF} and Problem~\ref{prob:PFTHIN}, respectively,
as obtained by setting 
\begin{equation}
 \theta= \frac{T_{\rm ref}}{T}
\end{equation}
instead of \eqref{eq:TTRANS}.
While the order parameter $\phi_1$ is representing the liquid fraction,
the order parameters $\phi_2,\dots, \phi_M$ are associated with certain solid states,
as, e.g., crystalline structures, of the given material.
With this in mind, the positive reference value $T_{\rm ref}\in \R$ is chosen to be the melting temperature
and we set $T_{\rm ref} = T_1 = \cdots = T_M = 1$. 
Accordingly we set $L_1=0$ and $L_2=\cdots = L_M > 0$ in all our computational examples.

Efficient step sizes $\rho_\nu$ for 
the Schur--Newton iteration~\eqref{eq:descent_method}
as required in  Theorem~\ref{thm:schur_newton_convergence} are determined by bisection. 
The iteration is stopped, once the
criterion
\begin{align}\label{eq:nsn_termination}
    \tfrac{\Vert \theta^{\nu+1} - \theta^{\nu} \Vert_\theta}{\Vert \theta^{\nu}\Vert_\theta} \leq \tolcorrection
\end{align}
is satisfied. 
We use $\tolcorrection = 10^{-11}$ in all our computations.

For each time step a grid hierarchy is obtained
either by uniform refinement or
according to the adaptive coarsening and refinement strategy described in Subsection~\ref{subsubsec:adaptivity}.

The initial iterate for the algebraic Schur--Newton solver is derived by nested iteration,
i.e.,  on each refinement level an initial iterate is obtained by nodal interpolation 
of the final iterate from the preceding one. 
On the first refinement level, the initial iterate 
is obtained by nodal interpolation of the final approximation 
in the preceding time step.
For the first time step, the continuous initial conditions 
are interpolated to the initial grid $\cTold$.

All numerical experiments were conducted using the DUNE (Distributed and Unified Numerics Environment)
framework and the DUNE-modules \emph{dune-subgrid} and \emph{dune-tnnmg} (cf. ~\cite{bastian_et_al:dune2:2008,bastian_et_al:dune1:2008,graser2009dune}).

\subsection{Experimental order of convergence}\label{subsec:EOC}
In order to numerically assess the spatial discretization error 
of the finite element discretization stated in Problem~\ref{prob:PFSPATDISC},
we consider the multi-phase Penrose--Fife Problem~\ref{prob:PF},
with $\Omega=(0,2)^2\subset \R^2$, $M=5$ phases
of which only the liquid and one solid phase is present,
and the following parameters
\begin{align*}
    \eps &= 6\cdot 10^{-2}, &
    c_v &= 1, &
    q &= 0, &
    \conv &= 0, \\
    \kappa &= 1, &
    \kin &= 1, & L_1 &= 0, & L_\a &= 2, \; \alpha =2,\dots,M.
\end{align*}
We select the  initial temperature $\theta^0=0.5$.
The initial phase field $\phi^0$ is given by 
\begin{align*}
  \phi^0_2(x) &=  
  \begin{cases}
    1 & \textrm{ if } d(x) < 0.5 \\
    \vert\tfrac12\cos(5\pi(d(x)-0.5))+0.5\vert & \textrm{ if } 0.5 \leq d(x) < 0.7  \\
    0 & \textrm{ else }
  \end{cases}
\end{align*}
where $d(x)$ stands for the Euclidean distance from $x$ to $(1,1)$,
$\phi^0_1(x) = 1 -  \phi^0_2(x,t)$, and $\phi^0_\alpha = 0$ for $\alpha = 3, 4, 5$,
as depicted in Figure~\ref{fig:initialvalue}.
We select the uniform time step size $\tau = 5\cdot 10^{-4}$.
A sequence $\cT_0,\dots, \cT_{10}$ of grids
is obtained by uniform refinement of $\cT_0$ consisting of 
a partition of $\Omega$ into two triangles.

\begin{figure}
    \includegraphics[draft=false, width=0.31\textwidth, trim=0 140 0 0, clip]{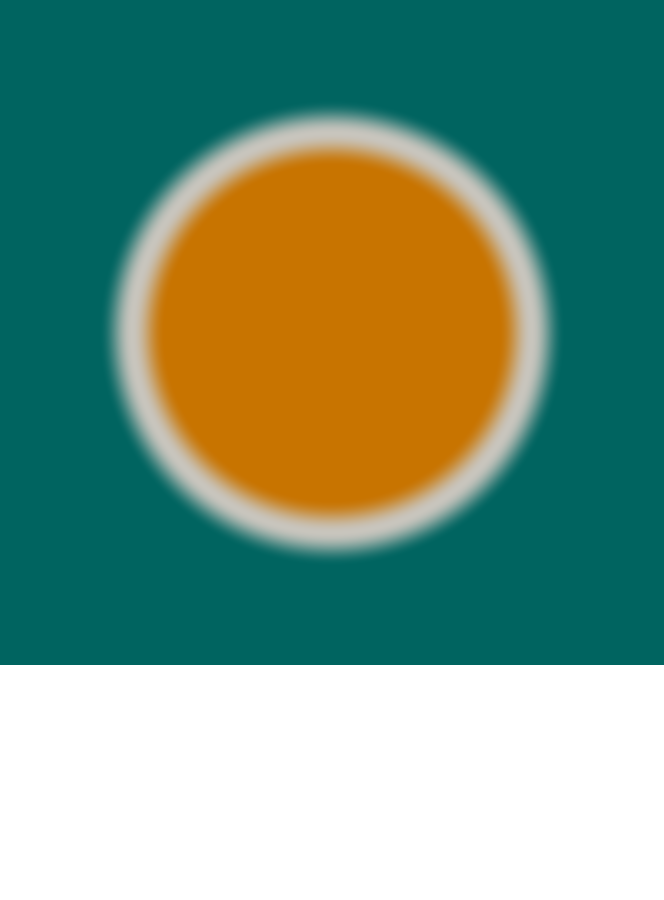}
    \caption{Initial phase field: A circular solid phase (orange) in a liquid environment (teal).}
    \label{fig:initialvalue}
\end{figure}

Figure~\ref{fig:disc-err} shows the approximate discretization error in the first time step
plotted over the mesh size $h_j$, $j=2,\dots,9$. 
The exact error is approximated by $e_\phi = \phi_\cT - \phi^*$ 
and $e_\theta = \theta_\cT - \theta^*$,
with approximations $\phi^*$ and $\theta^*$ obtained from $\cT_{10}$.
Our results suggest optimal order $O(h)$ of the discretization error
$\|e\| = \|e_\phi\|_\phi + \|e_\theta\|_\theta$.

\begin{figure}
 \input{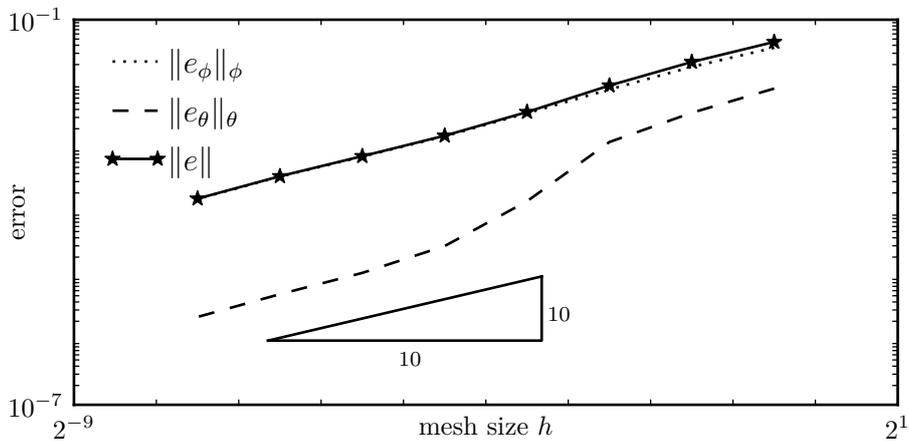}
 \caption{Discretization error $\|e\| = \|e_\phi\|_\phi + \|e_\theta\|_\theta$
     and its components $\|e_\phi\|_\phi$ and $\|e_\theta\|_\theta$ over mesh size.}
  \label{fig:disc-err}
\end{figure}

We next investigate the convergence properties of the non-smooth Schur--Newton method
as applied to the discrete saddle point problem in the first time step.
Figure~\ref{fig:rateItSteps} depicts the number $\nu_{\textrm{max}}$
of iteration steps needed
until the stopping criterion 
\eqref{eq:nsn_termination} is satisfied 
plotted over $N=\op{dim}\cS_j$, where $\cS_j$ is the finite element space associated with $\cT_j$.
The results indicate  mesh independence of 
the Schur--Newton iteration. While up to $\nu_{\textrm{max}}=17$ iteration steps
are required on coarser levels, 
only $\nu_{\textrm{max}}\leq 7$ steps are needed
once the diffuse interface is properly resolved by sufficiently fine grids.
This is in accordance with previous computations with multi-component
Cahn--Hilliard systems~\cite{graeser2014}.

We also computed the approximate solution for the first 500 time steps 
utilizing the grid $\cT_7$ obtained by seven uniform refinements 
to illustrate the evolution of the approximate entropy 
as depicted in Figure~\ref{fig:circleEntropy}.
\begin{figure}
  \input{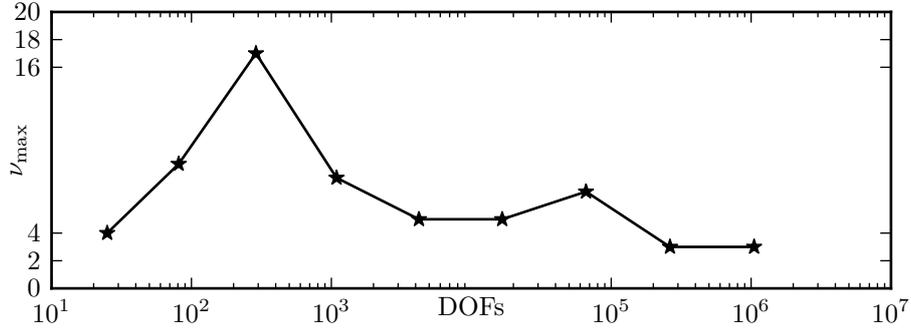}
  \caption{Number of Schur--Newton iterations $\nu_{\textrm{max}}$ needed to solve Problem~\ref{prob:PFSPATDISC} over $N$.}
  \label{fig:rateItSteps}
\end{figure}

\begin{figure}
  \input{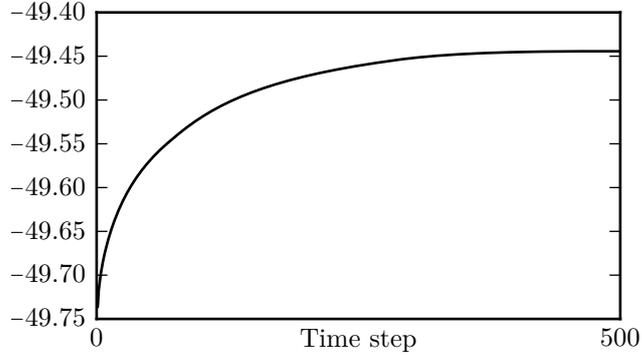}
  \caption{Approximate entropy over time steps.}
  \label{fig:circleEntropy}
\end{figure}

\subsection{Evolution of energy and entropy}
In order to illustrate the equilibration of energy in terms of latent heat and temperature
and the evolution of entropy, we consider the multi-phase Penrose--Fife Problem~\ref{prob:PF}
on the unit square $\Omega=(0,1)^2\subset \R^2$ with  $M=5$ phases 
of which only the liquid and one solid phase are present, and the parameters
  \begin{align*}
      \eps &= 8\cdot 10^{-2}, &
          c_v &= 1, &
      q &= 0, &
          \conv &= 0, \\
          \kappa &= 1, &
      \kin&=1, & L_1 &= 0, & L_\a &= 2,\; \a=2,\dots,5.
  \end{align*}
We choose an initial configuration with two phases (liquid and solid) 
and a planar interface according to
  \begin{align*}
      \phi^0_2 (x_1, x_2) = 
    \begin{cases}
      1 & \textrm{ if } x_1 > x^\star + 0.1, \\
      10(x_1 - x^\star)  & \textrm{ if } x^\star +0.1  \geq x_1 > x^\star \\      
      0 &  \textrm{ if }  x^\star \geq x_1,
    \end{cases} ,
  \end{align*}
$\phi^0_1 = 1- \phi^0_2$, and $\phi^0_\alpha = 0$, $\alpha = 3,4,5$.
The parameter $x^\star$ and constant initial temperature $\theta^0 = \theta^\star$
will be fixed later.
  
We select the time step size $\tau = 5\cdot 10^{-3}$.
The grid $\cT$ is obtained by eight uniform refinements
of an initial partition of $\Omega$ into two triangles.

The evolution of temperature is illustrated in terms of its maximal variation
\begin{align*}
    \theta_d^n = \max_{x\in\Omega}\theta^n(x) - \min_{x\in\Omega}\theta^n(x), 
    \qquad n = 1,\dots, 500,
\end{align*} 
and the average
\begin{align*}
    \theta_m^n = \tfrac12 (\max_{x\in\Omega}\theta^n(x) + \min_{x\in\Omega}\theta^n(x)), 
    \qquad n = 1,\dots, 500,
\end{align*} 
of its extremal values. 
The parameter $x^\star$ and constant initial temperature $\theta^0 = \theta^\star$
are set to  $x^\star=0.8$ and $\theta^\star = 0.2^{-1}$ in our first experiment
and to $x^\star=0.2$ and $\theta^\star = 1.5^{-1}$ in our second experiment.
The corresponding two evolutions are illustrated in Figure~\ref{fig:feedback01} 
and Figure~\ref{fig:feedback02}, respectively. 
Both figures show several time steps of the phase field in the left picture.
As the solution is constant in vertical direction only a cut-out is shown.
The picture on the right shows the evolution 
of temperature in terms of $\theta_m^n$ and $\theta_d^n$ and of
the entropy $S=\hat{S}(\theta^n, \phi^n)$ (cf. \eqref{eq:S_hat}) approximated by numerical quadrature.

In the first experiment we observe a growth of the initial solid grain
that slows down continuously due to intrinsic specimen heating by solidification.
Conversely, the shrinking of the initial grain observed in the second experiment
slows down due to intrinsic cooling by melting.
In both experiments,
absorption or release of latent heat
is driving the approximate temperature $\tfrac1{\theta^n}$ towards
the melting temperature $T=1$ at equilibrium.
Both cases exhibit a monotonically increasing entropy.

\begin{figure}
  \newcommand{\strip}[1]{\includegraphics[width=0.8\textwidth, trim={0 550 0 0}, clip, frame]{#1}\newline}
  \begin{minipage}{0.39\textwidth}
      \vspace{10pt}
      \strip{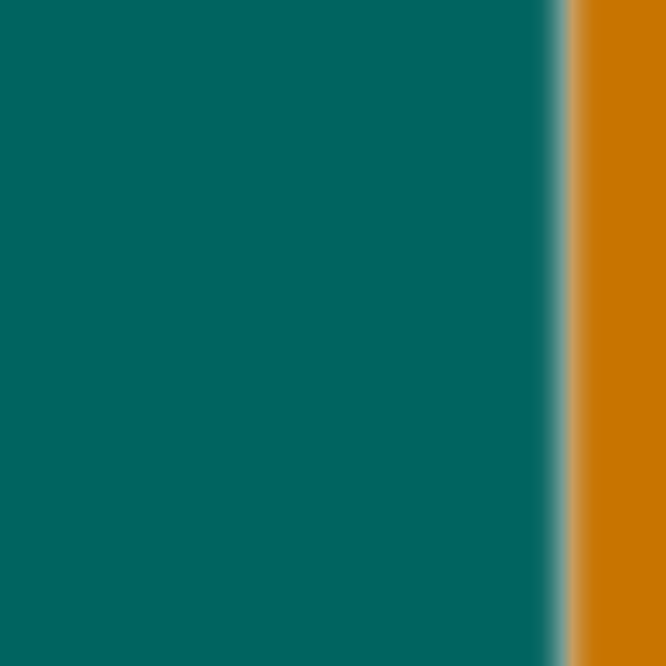}
      \strip{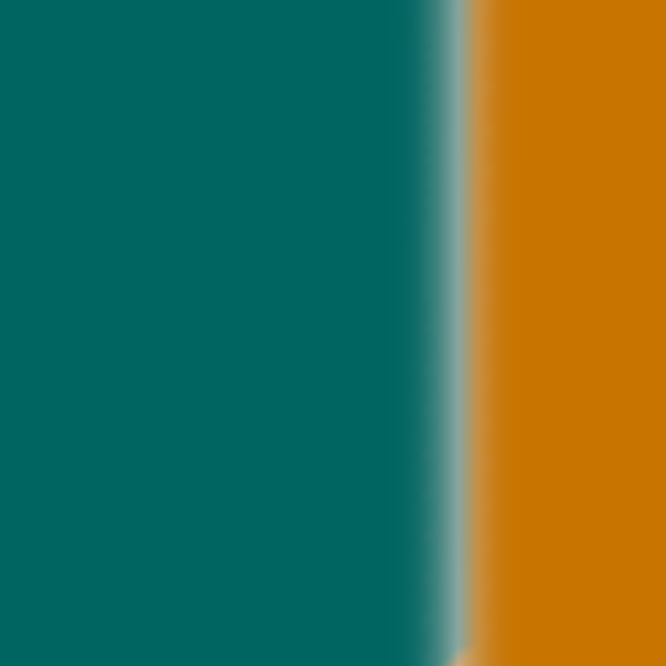}
      \strip{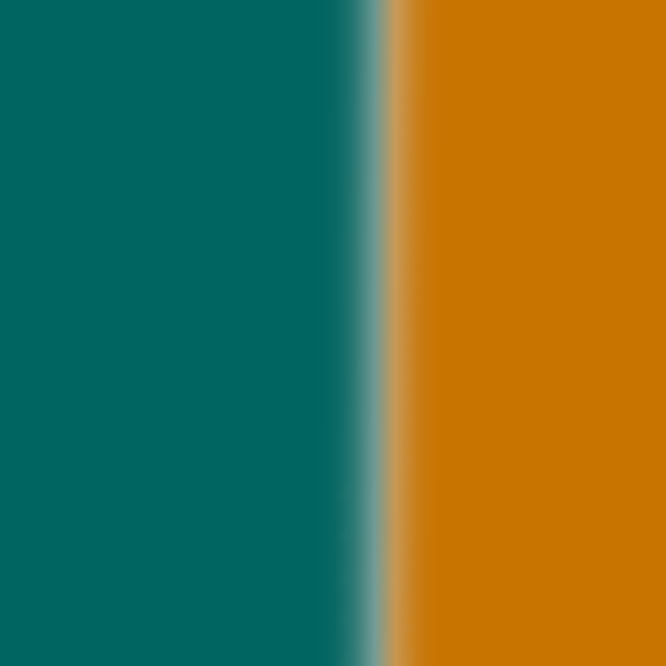}
      \strip{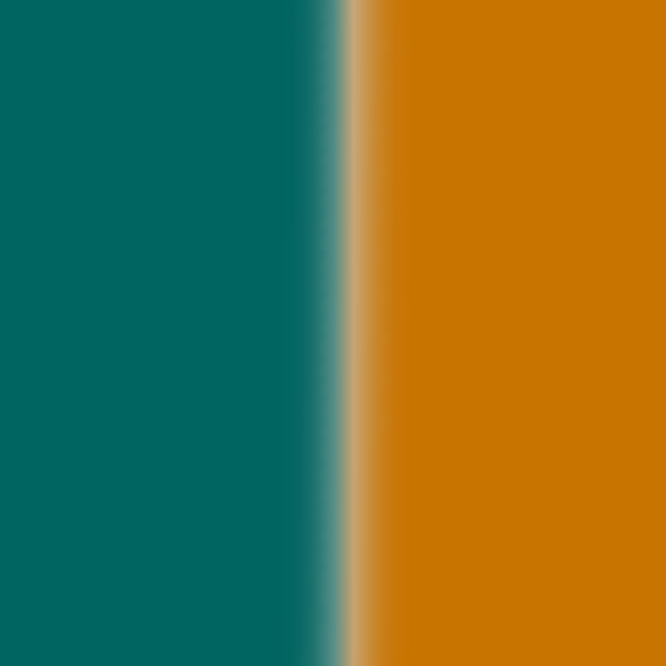}
      \strip{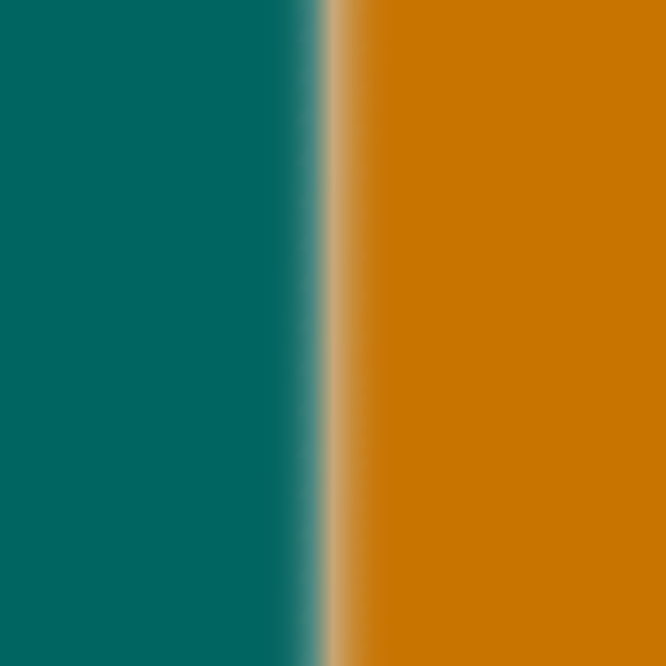}
      \strip{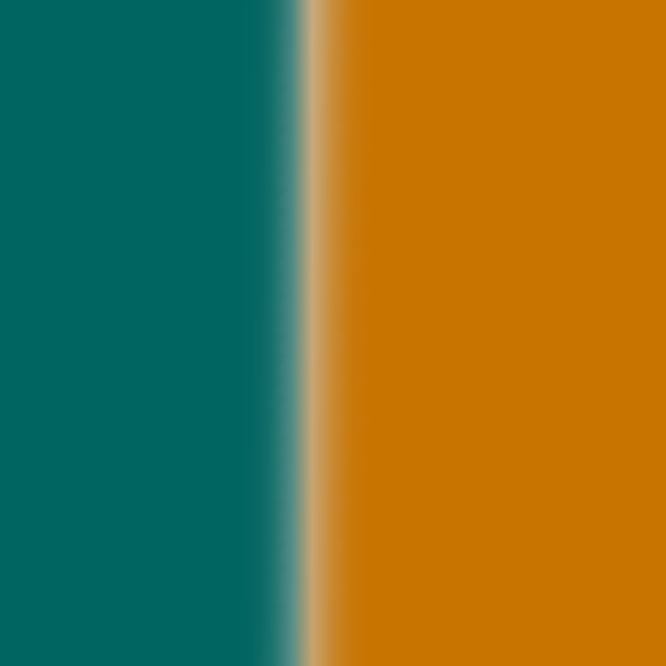}
  \end{minipage}%
  \hfill%
  \begin{minipage}{0.59\textwidth}
      \input{img/growth.pgf}
  \end{minipage}
  \caption{Solidification induced by latent heat. Left: Evolution of phases  by means of approximations at the time steps 0, 40, 80, 120, 160, 500. Right: Evolution of inverse temperature $\theta$ and approximate entropy $S$.}
  \label{fig:feedback01}
\end{figure}

\begin{figure}
  \newcommand{\strip}[1]{\includegraphics[width=0.8\textwidth, trim={0 550 0 0}, clip, frame]{#1}\newline}
  \begin{minipage}{0.39\textwidth}
      \vspace{10pt}
      \strip{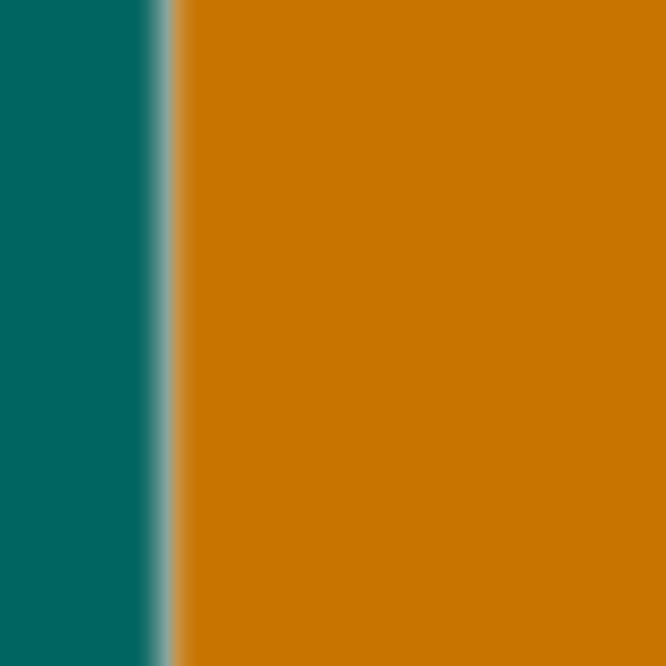}
      \strip{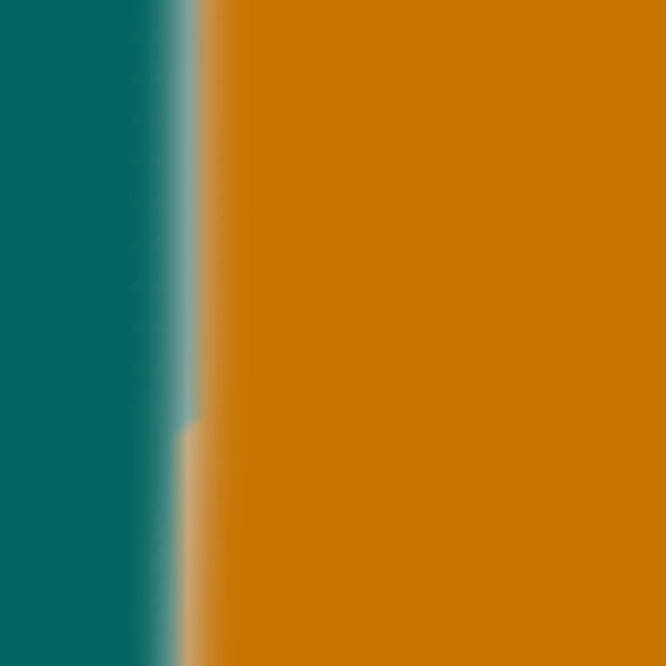}
      \strip{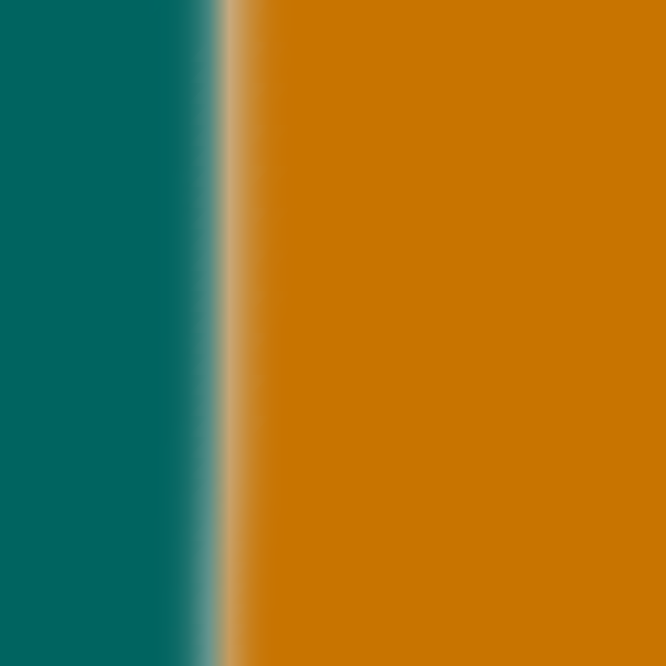}
      \strip{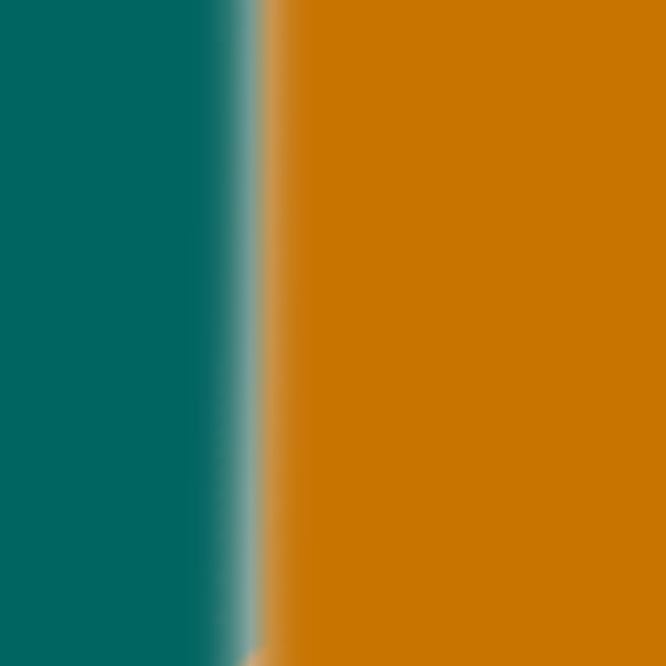}
      \strip{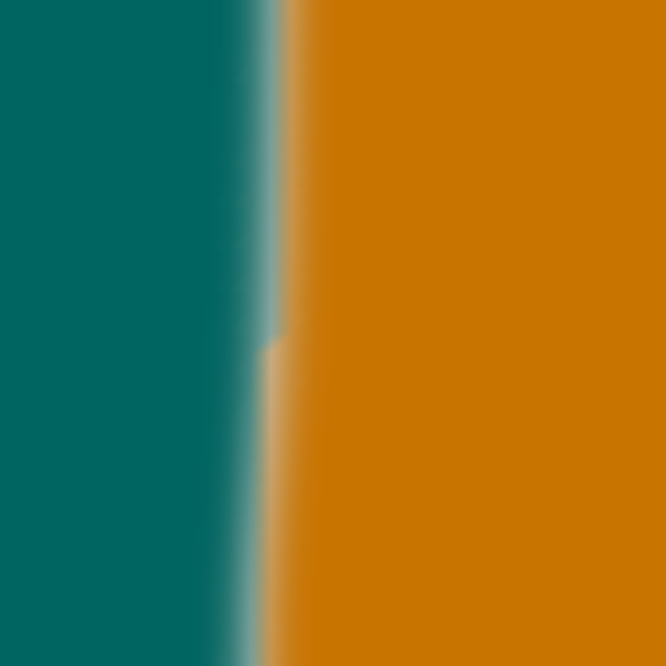}
      \strip{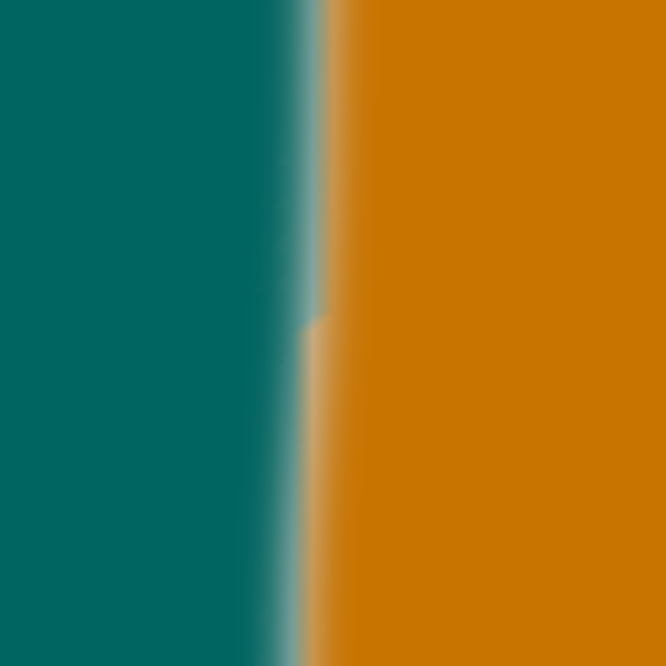}
  \end{minipage}%
  \hfill%
  \begin{minipage}{0.59\textwidth}
      \input{img/shrink.pgf}
  \end{minipage}
  \caption{Melting induced by latent heat. Left: Evolution of phases  by means of approximations at the time steps 0, 40, 80, 120, 160, 500. Right: Evolution of inverse temperature $\theta$ and approximate entropy $S$.}
  \label{fig:feedback02}
\end{figure}

\section{A liquid phase crystallization process}
\label{s:application}

% description of LPC
Liquid phase crystallization (LPC) is an emerging technology
to produce silicon thin film solar cells with advanced photoelectronic properties
that enable high efficiency devices.
In an LPC process, silicon is deposited on a substrate 
and then is  swept over with a heat source 
for local melting and subsequent recrystallization
to coarser, photoelectronically beneficial structures.
Optimization of parameters like speed, shape or intensity of the heat source
for various semi-conducting materials 
is the subject of current experimental research,
cf., e.g., \cite{amkreutz2011electron, eggleston2012large, kuhnapfel2015towards}.

% descripton of problem
Mathematical modelling of LPC processes can be performed in the 
framework of multi-phase field models presented in Section~\ref{ss:assumptions}.
To this end, we consider the thin-film approximation Problem~\ref{prob:PFTHIN} 
on $\Omega = (0,2)^2$ 
with $M=5$ phases with $\phi_\a$, $\a=2,\dots M$, representing different crystal structures 
and the parameters
  \begin{align*}
   \eps &= 5\cdot 10^{-2}, &
   c_v &= 1, &
   T_\Gamma &= 0.1, &
   T_\a &= 1, \;  \a=2,\dots, M,
   \\
   \conv' &= 5\cdot 10^2, &
   \kappa &= 1, &
    L_1 &= 0, & 
      L_\a &= 1, \;  \a=2,\dots, M .
\end{align*}

In order to prescribe a slower solid--solid interface evolution
in comparison the to solid--liquid interfaces,
we now choose a solution-dependent kinetic coefficient 
$\kin = (\kin_\alpha(\phi, \nabla\phi))_{\alpha=1}^N$ 
according to 
\begin{align*}
    \kin_\alpha(\phi_1, \dots, \phi_M, \nabla\phi_1, \dots, \nabla\phi_M) &= \begin{cases}
        100 & \vert\phi_1\nabla\phi_\alpha - \phi_\alpha\nabla\phi_1\vert < 10^{-5},\\ 1 & \textrm{else}.
    \end{cases}
\end{align*}

The heat source is represented by
\begin{align*}
    q(\theta, x, t) &=  q_{\op{max}}(\theta) \cdot\exp\left(-\tfrac{\lvert x_1 - q_p(t)\rvert}{2q_w^2}\right), \\
    q_{\op{max}}(\theta) &= \tfrac1\theta h_q (1-\tfrac{\theta_q}{\theta}) + \conv'(\tfrac1{\theta} - T_\Gamma)
\end{align*}
and we select the parameters
\begin{align*}
    q_p(t) &= 0.9 + 1.5t, &
    q_w &= 0.2, &
    \theta_q &= \tfrac18 &
    h_q &= 5\cdot 10^2.
\end{align*}
Notice that the consistency statements of Proposition~\ref{prop:THERMOCONSISTENCY} and~\ref{prop:THERMOCONSISTENCY-disc}
cannot be applied in this case, 
because $q\neq 0$ and because the coefficient $\kin$ depends on $\phi$.
It is unclear if similar results can be shown for solution dependent coefficients.

Observe that the heat source peaks at $x_1=q_p(t)$ and is moving across the device
from left to right with constant speed.
The initial temperature is given by $\theta^0=10=\tfrac1{T^0} = \tfrac1{0.1}$
and the  initial phase configuration $\phi^0$ 
is depicted in the upper left picture of Figure~\ref{fig:lpc01}
with teal color representing the liquid phase.
We choose an initial value  that already prescribes a local liquid phase, 
because this simplifies to exclude superheating effects such as
instantaneous global melting of the whole material.

\begin{figure}
    \newcommand{\strip}[1]{\includegraphics[width=0.45\textwidth, trim={15 250 565 140}, clip]{img/lpc-evolution#1-mesh.png}}
   \strip{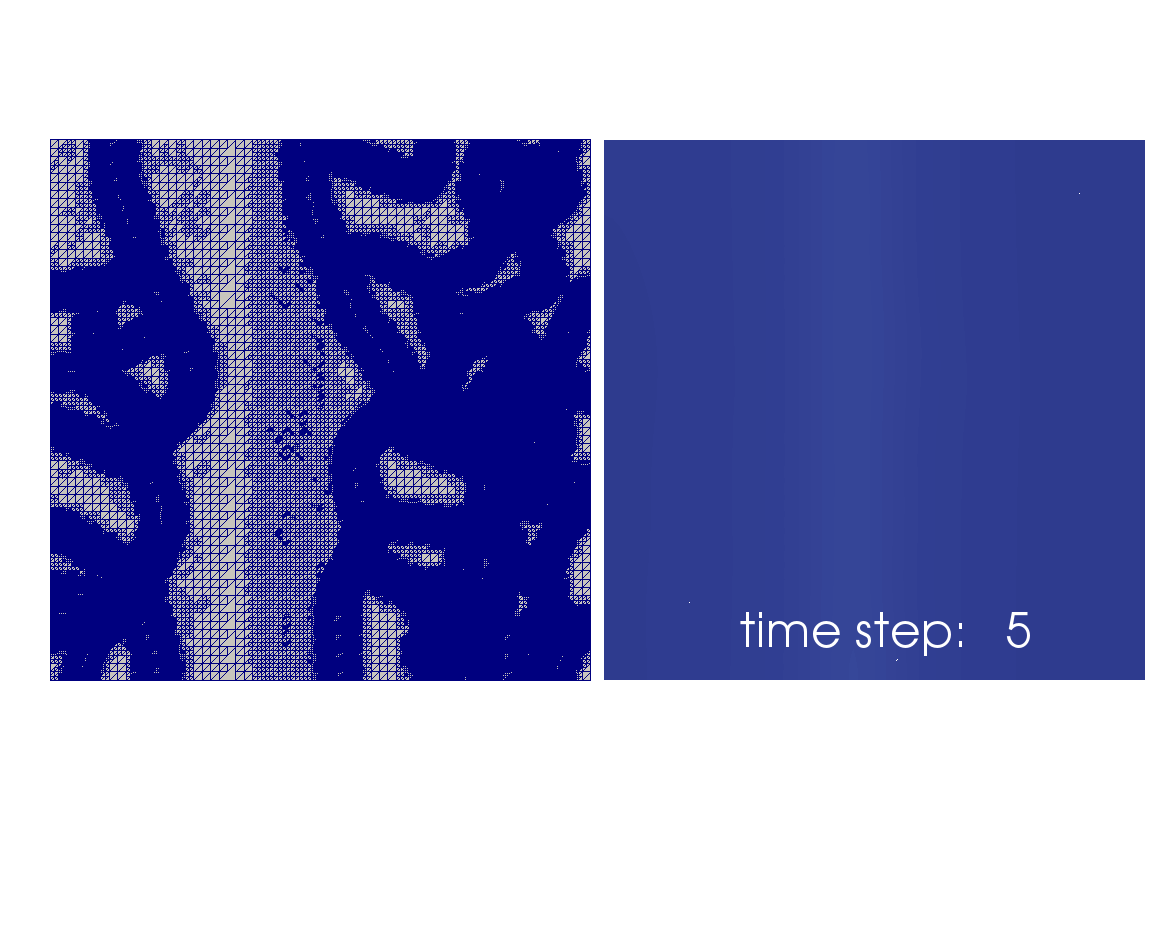}
   \strip{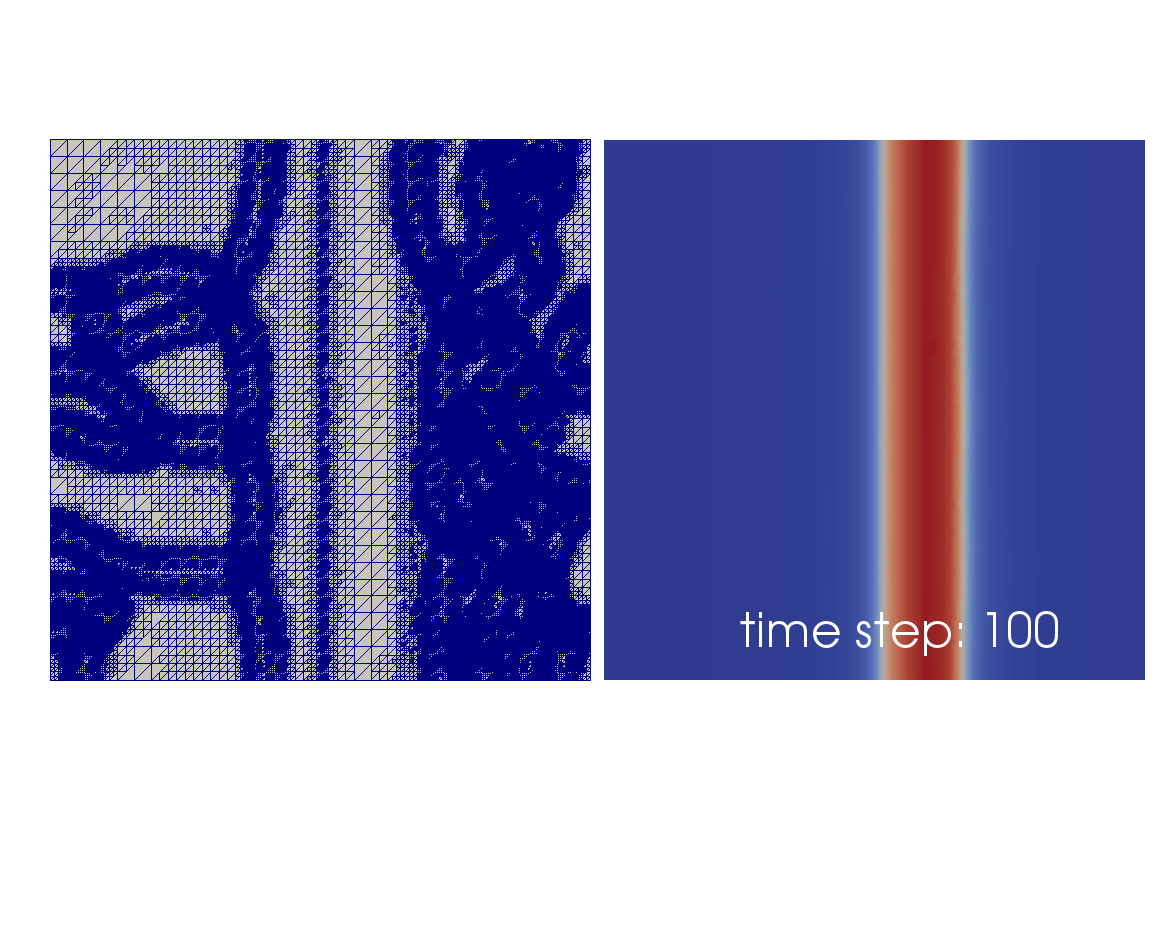}
    \caption{Adaptively refined grid for time step 5 (left) and 100 (right).}
    \label{fig:lpc_mesh}
\end{figure}

\begin{figure}
  \newcommand{\strip}[1]{\includegraphics[width=0.45\textwidth, trim={15 250 15 140}, clip]{img/lpc-evolution#1.png}}
  \strip{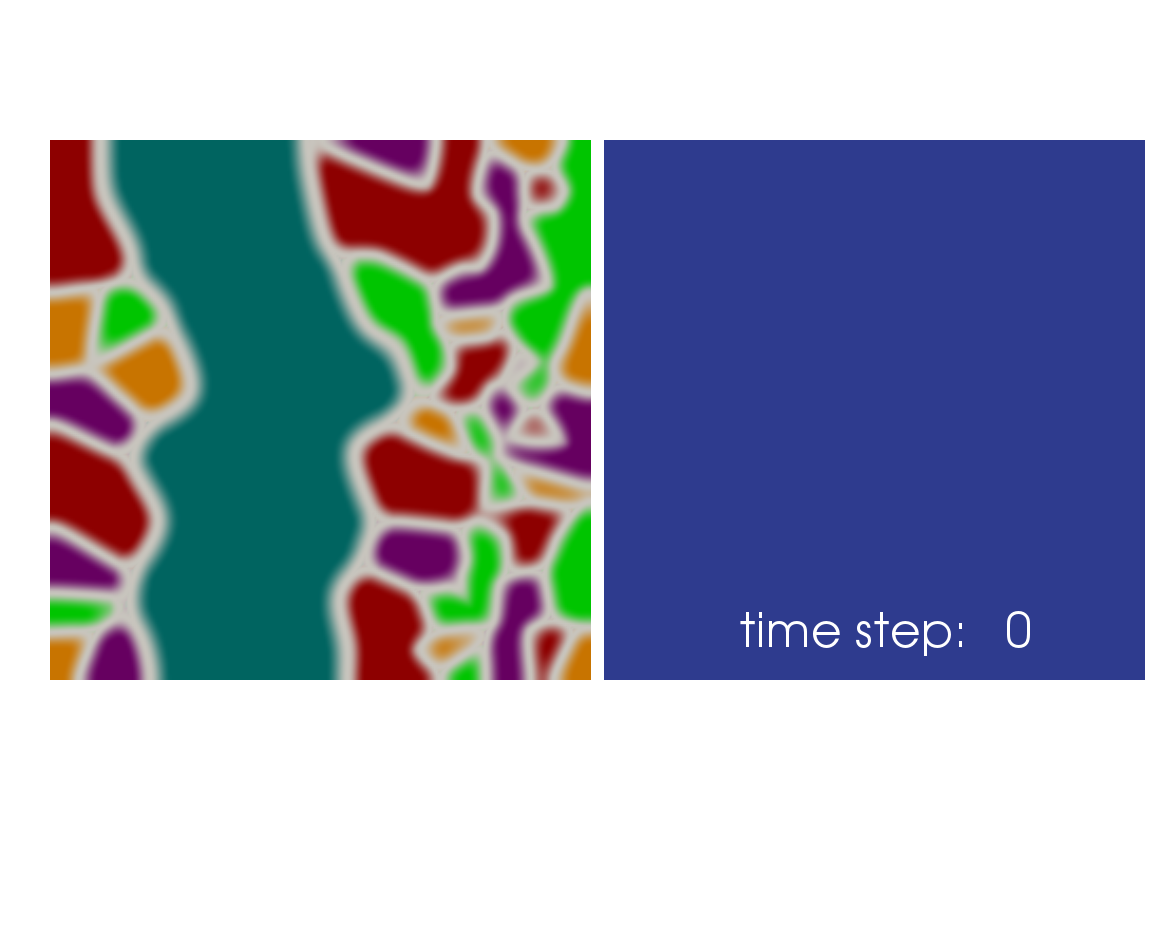}
  \strip{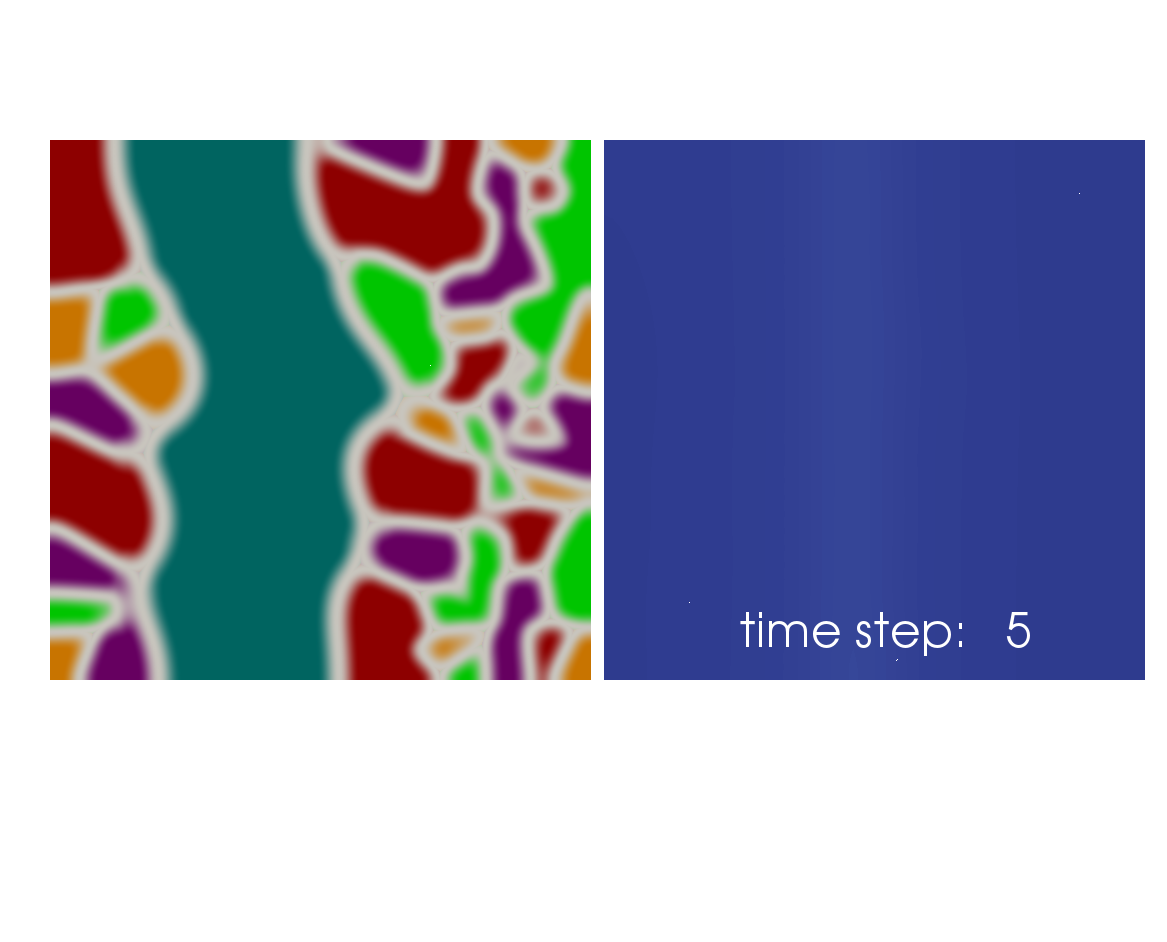}
  \strip{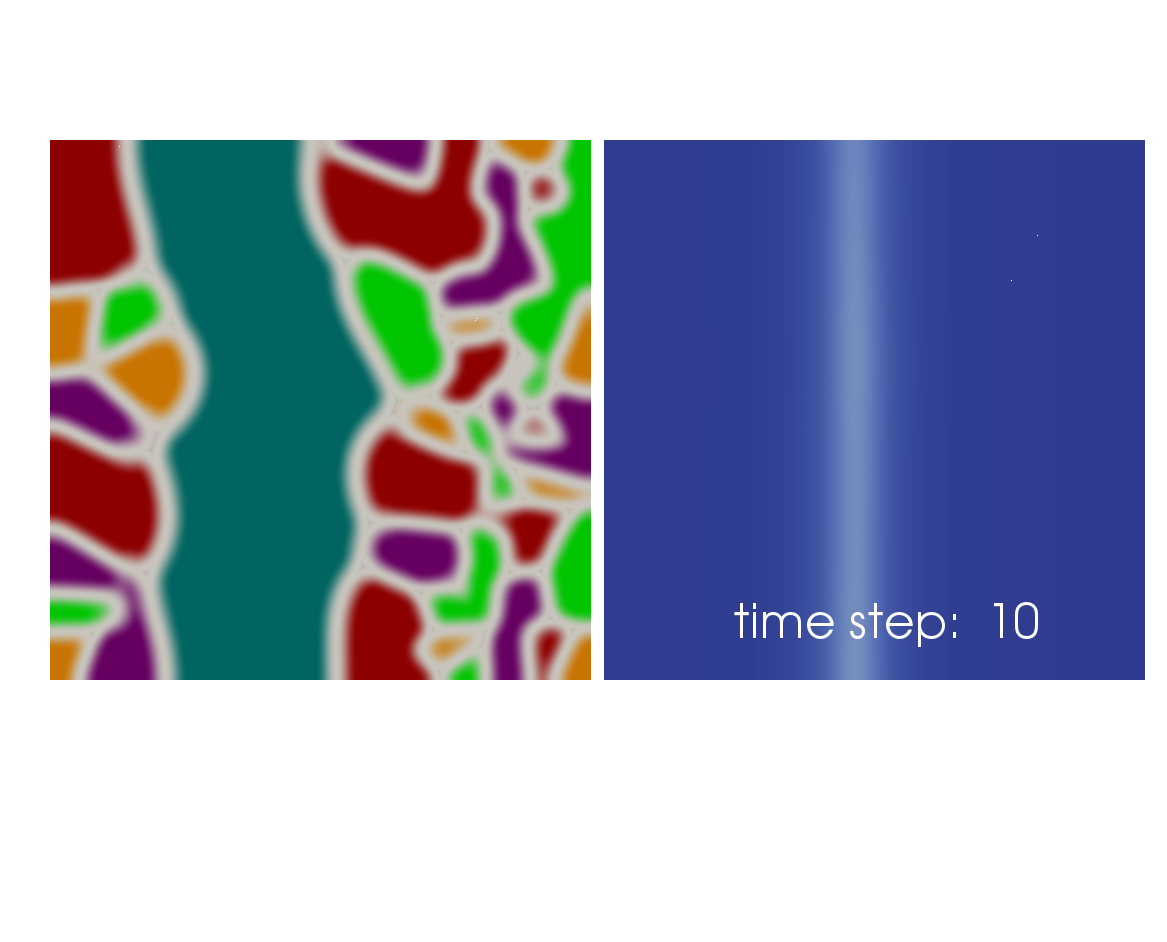}
  \strip{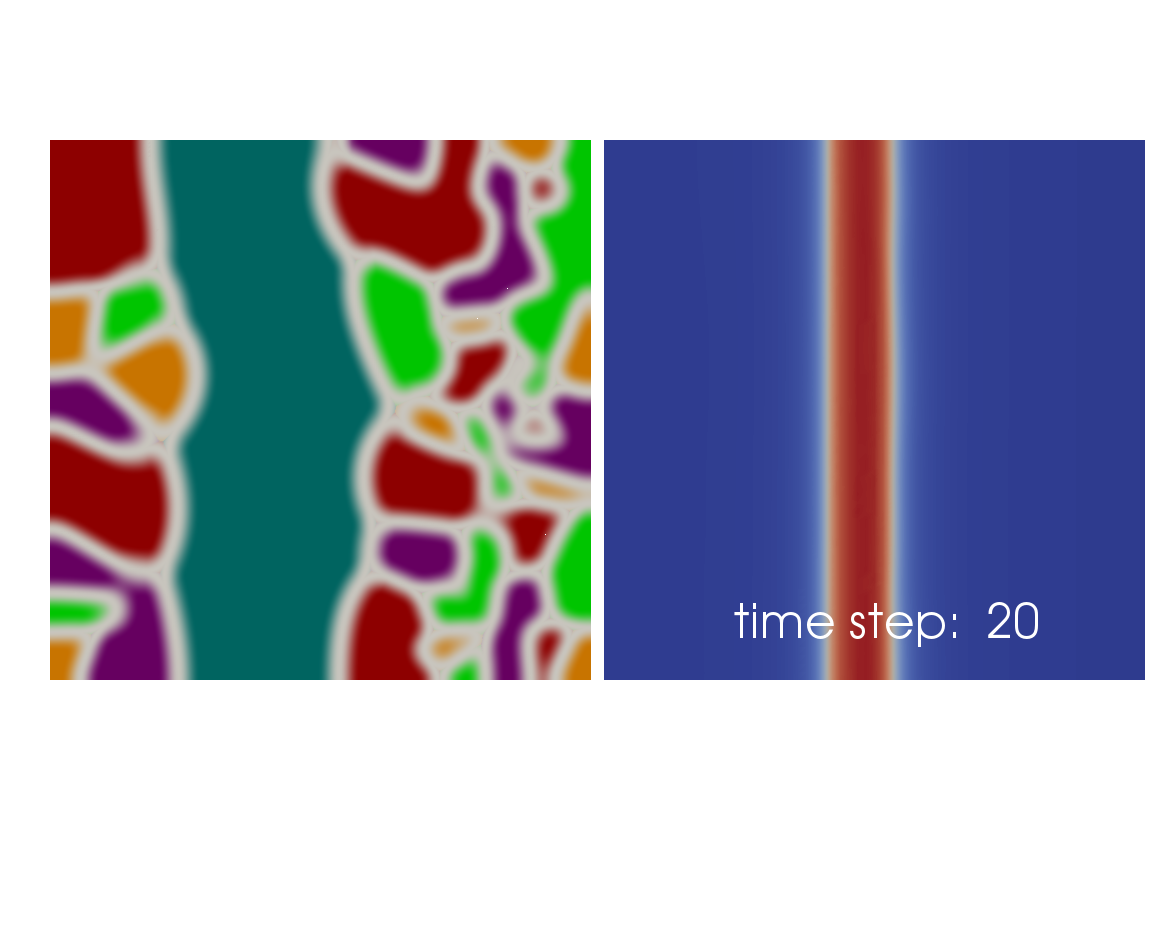}
  \strip{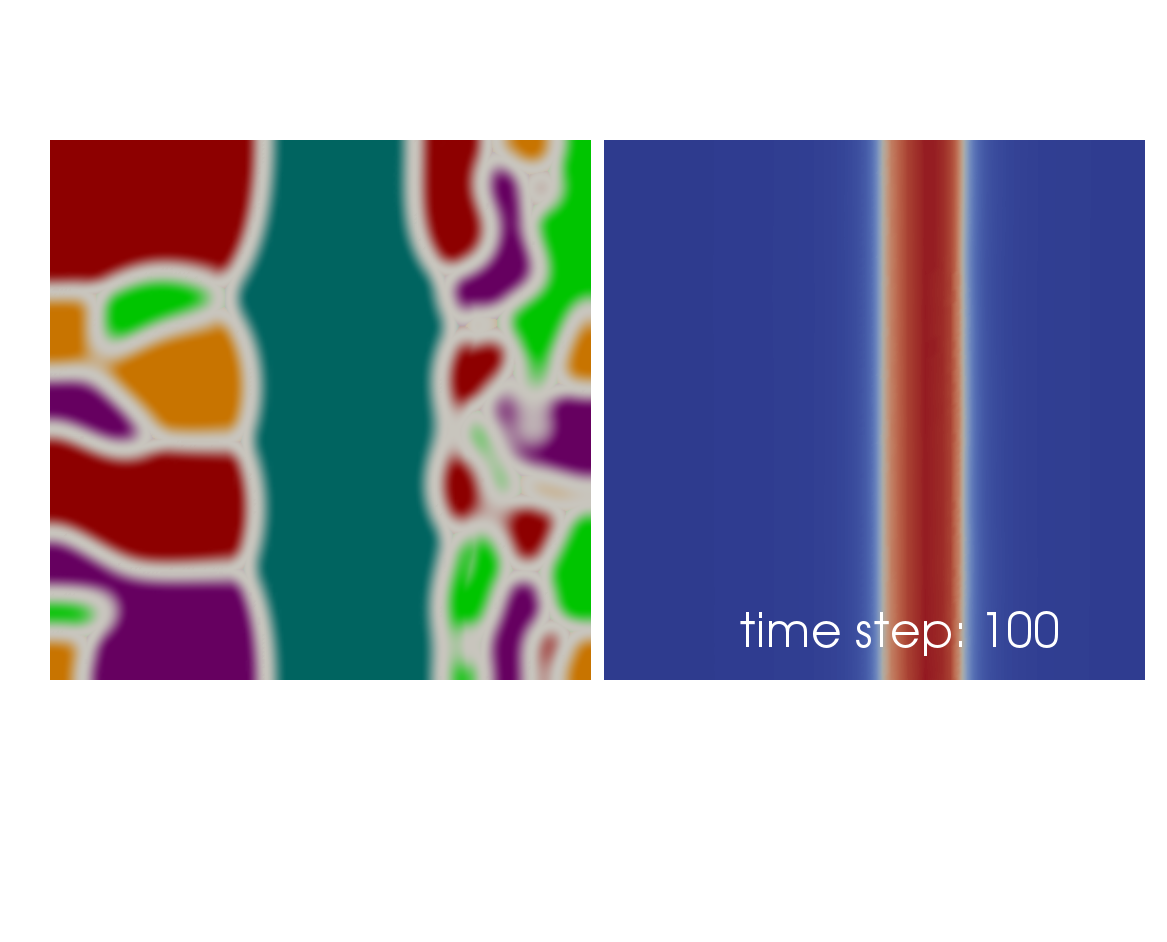}
  \strip{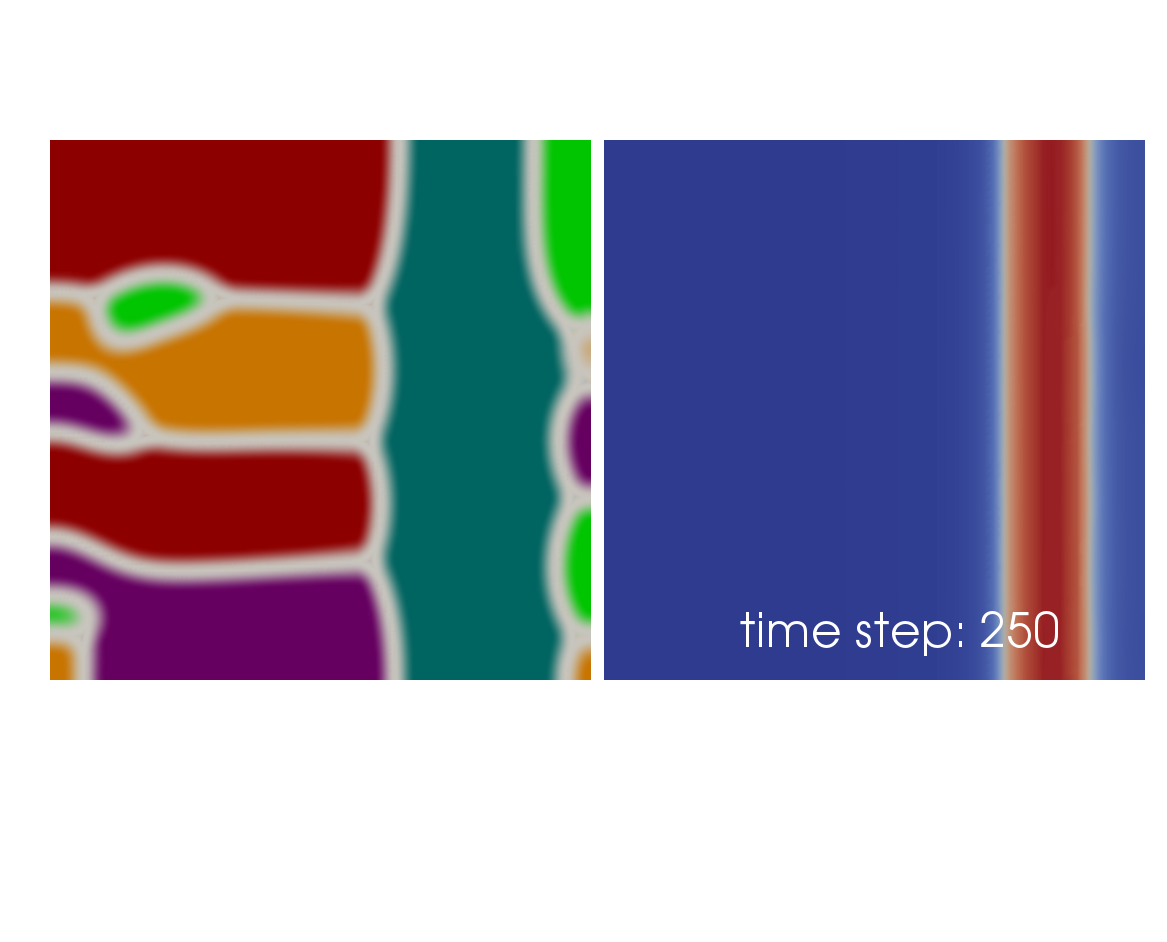}
  \caption{%
      Initial distribution of phases and temperature (top left)
      and approximate distribution of phases and temperature for the time steps 5, 10, 20, 100, 250.}
  \label{fig:lpc01}
\end{figure}

% descripton of discretization
We select the time step size $\tau = 2\cdot 10^{-3}$.
In each time step,
we construct a sequence of locally refined meshes $\cT_0, \dots, \cT_J$
using the adaptive refinement algorithm described in Subsection~\ref{subsubsec:adaptivity}.
The derefinement and refinement parameters are selected as follows
\begin{align*}
    j_{\op{min}}&= 2, & 
    \tolcoarsen &= 10^{-6}, &
    \rho &= 0.9, & 
    \toladapt &= 8\cdot10^{-3}.
\end{align*}
In the first time step, we start with an initial grid $\cTold$ obtained 
by eight uniform refinements of an initial partition of $\Omega$ into two triangles.
The final mesh $\cT$ for time steps 5 and 100 is depicted in Figure~\ref{fig:lpc_mesh}.
In both cases, the mesh is obtained by 6 adaptive refinement steps after 
coarsening.

% descripton of computatio
%%% parameters as given by data/lpc-evolution.parset (see section solver)
%%% logfile as given by data/lpc-evolution.cout (maybe apply grep -v BMG | grep -v "\-\-\-") 
%%%     where TS: time solver, SS: spatial solver, SN: Schur--Newton
\begin{figure}
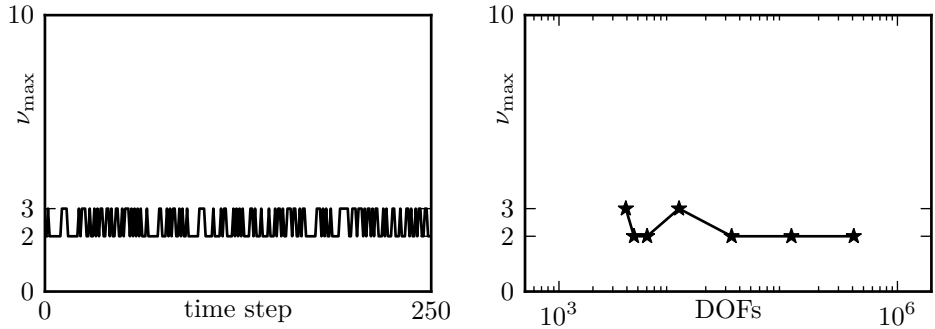

  \begin{minipage}{0.49\textwidth}
      \input{img/rateLPC.pgf}
  \end{minipage}
  \begin{minipage}{0.49\textwidth}
      \input{img/rateLPC2.pgf}
  \end{minipage}
  \caption{
  Robustness of Schur--Newton convergence.
  Left: $\nu_{\textrm{max}}$ corresponding to the final mesh over time steps.
  Right: $\nu_{\textrm{max}}$ corresponding to different adaptively refined meshes
  over degrees of freedom $N$ for the fixed time step $5$.}
  \label{fig:rateLPC}
\end{figure}

Figure~\ref{fig:lpc01} shows the (approximate) evolution of phases and temperature.
For each of the time steps 0, 5, 10, 20, 100, 250
the left picture depicts the liquid and the different crystalline  phases
while the right picture shows the temperature distribution.
The liquid phase adapts to the shape of the heat source
in course of the evolution.
As the heat source travels on,
the right hand crystalline phases start melting,
while recrystallization occurs on the left solid--liquid interfaces,
because the local temperature drops below melting temperature.
Note that  recrystallization leads to coarser grain structures
which is a  characteristic feature of LPC.

To briefly highlight the efficiency of the Schur--Newton method with nested iteration,
the number of Schur--Newton steps $\nu_{\textrm{max}}$
required on the finest mesh is plotted over the time step in the left of Figure~\ref{fig:rateLPC}.
We observe that this number does not exceed 3 in any time step.
Mesh independence is illustrated by the number of Schur--Newton steps
over $N=\op{dim}\cS_j$ in time step $5$.

% descripton of computatio
%%% parameters as given by data/lpc-evolution.parset (see section solver)
%%% logfile as given by data/lpc-evolution.cout (maybe apply grep -v BMG | grep -v "\-\-\-") 
%%%     where TS: time solver, SS: spatial solver, SN: Schur--Newton

\section*{Acknowledgment}
The work was supported by the Helmholtz Virtual Institute
HVI-520 ``Microstructure Control for Thin-Film Solar Cells'',
the Einstein Foundation Berlin via the research center {\sc Matheon},
and by the Sino-German Science Center (grant id 1228) on the occasion of 
the Chinese-German Workshop on Computational and Applied Mathematics 
in Augsburg 2015.
\bibliographystyle{abbrv}
\bibliography{paper}

\end{document}